\documentclass{amsart}
\usepackage{graphicx}
\usepackage{geometry} 
\usepackage{amssymb}

\usepackage[alphabetic,abbrev]{amsrefs}  
\renewcommand\MR[1]{} 

\usepackage[all]{xy}
\usepackage{xcolor}
\usepackage[hidelinks]{hyperref}
\usepackage{cleveref}
\usepackage{bm} 
\usepackage{tikz}
\usetikzlibrary{cd}
\numberwithin{equation}{section}

\newtheorem{theorem}[subsubsection]{Theorem}
\newtheorem{proposition}[subsubsection]{Proposition}

\newtheorem{corollary}[subsubsection]{Corollary}
\newtheorem{lemma}[subsubsection]{Lemma}

\newtheorem*{theorem*}{Theorem}

\newtheorem{ntheorem}{Theorem}
\newtheorem{nconjecture}[ntheorem]{Conjecture}
\newtheorem{aconjecture}{Conjecture}

\theoremstyle{definition}
\newtheorem{definition}[subsubsection]{Definition}

\theoremstyle{remark}
\newtheorem{example}[subsubsection]{Example}
\newtheorem{remark}[subsubsection]{Remark}

\newcommand{\Rep}{\operatorname{Rep}}

\newcommand{\Hom}{\operatorname{Hom}}
\newcommand{\End}{\operatorname{End}}

\newcommand{\Stab}{\operatorname{Stab}}

\newcommand{\supp}{\operatorname{supp}}
\newcommand\rad{\operatorname{rad}}
\newcommand\Span{\operatorname{span}}
\renewcommand\top{\operatorname{top}}

\renewcommand{\cong}{\simeq}

\newcommand\tfor{\text{for }}

\newcommand{\blambda}{\bm{\lambda}}

\newcommand\Ann{\operatorname{Ann}}
\newcommand\Ext{\operatorname{Ext}}
\newcommand\stHom{\underline{\operatorname{Hom}}}
\newcommand\stAut{\underline{\operatorname{Aut}}}
\newcommand\stEnd{\underline{\operatorname{End}}}
\newcommand\uHom{\underline{\operatorname{Hom}}}

\newcommand\uEnd{\underline{\operatorname{End}}}
\newcommand\im{\operatorname{im}}
\newcommand\id{\operatorname{id}}

\newcommand\Aut{\operatorname{Aut}}
\newcommand\soc{\operatorname{soc}}

\newcommand\Mat{\mathrm{Mat}}

\newcommand{\Tilt}{\mathrm{Tilt}}

\newcommand{\mN}{\mathbb{N}}
\newcommand{\mZ}{\mathbb{Z}}

\newcommand{\cA}{\mathcal{A}}

\newcommand{\cC}{\mathcal{C}}

\newcommand{\cT}{\mathcal{T}}
\newcommand{\cP}{\mathcal{P}}

\newcommand{\cI}{\mathcal{I}}

\newcommand{\bk}{\Bbbk}

\newcommand{\mA}{\mathbb{A}}
\newcommand{\mF}{\mathbb{F}}
\newcommand{\mP}{\mathbb{P}}

\newcommand{\mG}{\mathbb{G}}
\newcommand{\mI}{\mathbb{I}}
\newcommand{\mJ}{\mathbb{J}}

\newcommand\one{\mathbf{1}}
\newcommand{\unit}{\mathbf{1}}

\newcommand\kk{\Bbbk}
\renewcommand\o{\otimes}
\newcommand\ol{\overline}
\newcommand{\tto}{\twoheadrightarrow}

\newcommand{\ulambda}{\underline{\lambda}}
\newcommand{\umu}{\underline{\mu}}

\title[Elementary $p$-group and Steinberg modules]{A conjecture on the tensor ideal for an elementary $p$-group generated by the restriction of a Steinberg module}
\author{Kevin Coulembier}
\address{School of Mathematics and Statistics, University of Sydney, Australia}
\email{kevin.coulembier@sydney.edu.au}
\author{Johannes Flake}
\address{Mathematical Institute, University of Bonn, Germany}
\email{flake@math.uni-bonn.de}
\date{\today}

\begin{document}

\begin{abstract}
In previous work (Coulembier--Flake 2024), the authors conjectured that the tensor product of an arbitrary finite-dimensional modular representation of an elementary abelian $p$-group with the biggest non-projective restricted Steinberg $SL_2$-module is a restricted tilting module. We showed that the validity of the conjecture would have interesting implications in the theory of tensor categories in positive characteristic, in particular, with respect to the classification of incompressible symmetric tensor categories, which is the subject of arguably the main open conjecture in the area. We present here some evidence for the conjecture to hold.
\end{abstract}

\maketitle

\section{Introduction}

A deep theorem by Deligne \cites{De1, De2} classifies symmetric tensor categories with a growth condition in characteristic $0$ in terms of affine groups in the category of supervector spaces. In positive characteristic, recent discoveries in \cites{BE, BEO, CEO, AbEnv} and references therein led to a conjectural analog \cite{BEO}*{Conjecture~1.4} of Deligne's theorem, which broadly speaking asserts that the same class of tensor categories corresponds to affine groups in the union of the so-called \emph{generalized Verlinde categories}. In \cite{CF}, when studying certain aspects of the conjecture, we established a connection with modular representations of elementary abelian $p$-groups. In particular, we formulated a conjecture on such representations, derived certain consequences from the conjecture regarding the theory of incompressible tensor categories, and verified the conjecture in the smallest examples.

To state the conjecture and explain the progress of the current paper in simple terms, let $\kk$ be an algebraically closed field of positive characteristic $p$. Let $r\ge1$ be a natural number. Fix an embedding of the elementary abelian $p$-group $E$ of order $p^r$ into the lower unitriangular matrices in $SL_2:=SL_2(\kk)$. Let $(T_i)_{i\ge0}$ be the set of indecomposable tilting modules of $SL_2$ over $\kk$, labeled by their highest $SL_2$-weight, restricted to $E$ via the fixed embedding. It can be seen that $T_i$ remains indecomposable over $E$ for $i<p^r$, that $T_{p^r-1}$ is projective over $E$, and that any restricted tilting module is a direct sum of restricted tilting modules $T_i$ with $i<p^r$ (\Cref{lem:restricted-tilting}). Set $S:=T_{p^{r-1}-1}$, the restricted $(r-1)$-th Steinberg module. Unless explicitly mentioned otherwise, we only consider finite-dimensional representations of groups, in particular $\Rep E$ is the tensor category of finite-dimensional $E$-representations over $\bk$.

\begin{nconjecture}[\cite{CF}*{Conjecture~5.4.2(1), see also Lemma~5.4.5}] \label{conj:intro} For any $X\in\Rep E$, the representation $S\o X$ is the restriction to $E$ of an $SL_2$-tilting module.
\end{nconjecture}

\Cref{conj:intro} thus predicts that the thick tensor ideal generated by $S$ is surprisingly small, in particular, it contains only finitely many indecomposable objects. For comparison, if one deforms $X$ slightly, to similar representations of $E$, the number of indecomposable objects in the corresponding ideal is at least the cardinality of $\bk$, see \Cref{Sec9} and \Cref{Sec8}. One could say that $X$ seems to behave like a representation induced from a subgroup of $E$ with finite representation type.  

There are various reformulations of \Cref{conj:intro} which we discuss in \Cref{sec:conjecture}. In particular, the reference to the group $SL_2$ can be avoided as follows: Let $V$ be any faithful two-dimensional representation of $E$, let $\cC_V$ be the pseudo-tensor subcategory of $\Rep E$ whose indecomposable objects are direct summands in tensor powers of $V$. Then $\cC_V$ has a unique smallest thick tensor ideal $\cI_V$ strictly containing the thick tensor ideal formed by all projective objects. \Cref{conj:intro} is equivalent to $\cI_V$ being a thick tensor ideal even in $\Rep E$. While there are known classification results on thick \emph{triangulated} subcategories with ideal closure \cites{BCR-thick-subcategories,BIK-stratifying-modular,BIKP-stratification} of the stable category of $E$, these results do not seem to help much for our problem. Essentially, they only allow us to reduce the problem to the category of periodic modules with a fixed one-point support, which is still wild due to \cite{wild-periodic-modules}.

\Cref{conj:intro} holds trivially if $r=1$. In \cite{CF}*{Remark 5.4.4(1)}, we observed that \Cref{conj:intro} is true for $p=r=2$, using the classification of indecomposable $E$-modules. We present here the following evidence for \Cref{conj:intro} in the general case:

\begin{ntheorem} \label{thm:A} $S\o X$ is a  restricted tilting module for all $X$ in the full subcategory of $\Rep E$ generated under tensor products, duals, $\Omega$-shifts, direct sums, and taking direct summands (\Cref{lem:closure}) by
\begin{itemize}
\item modules of Loewy length at most $2$ (\Cref{cor:ll-2-any-r}),
\item Carlson modules (\Cref{cor:Carlson}),
\item extensions of any two quotients of a uniserial $p$-dimensional representation with a certain support (\Cref{thm:S-tensor-ext-uniserials}),
\item modules whose support variety (or rank variety) does not contain a certain point depending on the above fixed embedding of groups (\Cref{rem:SX-projective}),
\item restrictions of simple, standard, or costandard $SL_2$-modules, and arbitrary repeated Frobenius twists of restrictions of $SL_2$-tilting modules (\Cref{lem:SL2-std}, \Cref{lem:SL2-simples}), and
\item restrictions of $SL_2$-modules in which all non-zero weight spaces have weight at most $p^r-1$ (\Cref{lem:restricted-sl-in}).
\end{itemize}
\end{ntheorem}

\begin{ntheorem}[\Cref{thm:cyclic-S-proj}] \label{thm:B} If $r=2$, then for any $X\in\Rep E$, every cyclic direct summand in $S\o X$ is a restricted tilting module.
\end{ntheorem}

For $p=r=2$, it is easy to see that the subcategory as described in \Cref{thm:A} is $\Rep E$ even without using the classification of indecomposable $E$-modules, yielding a new proof (\Cref{cor:extensions-p2}) of \Cref{conj:intro} in this case. For general $p,r$, we are not aware of good descriptions of how `large' the subcategory is. 

In \Cref{sec:prelim}, we recall some preliminaries. In \Cref{sec:conjecture}, we discuss various equivalent versions of \Cref{conj:intro} and generalities on the conjecture. In \Cref{sec:analyse}, we establish fundamental properties of restricted tilting and Steinberg modules. In \Cref{sec:ext}, we study the effect of tensoring with Steinberg modules on certain extensions. In \Cref{sec:Carlson} and \Cref{sec:Loewy}, we prove the results in \Cref{thm:A} for Carlson modules and for modules of Loewy length $2$.
In \Cref{sec:cyclic}, we prove \Cref{thm:B}.


\section{Preliminaries} \label{sec:prelim}

For the entire paper, we fix an algebraically closed field $\bk$ of characteristic $p>0$.

\subsection{Monoidal categories}

By a {\bf pseudo-tensor category over $\bk$} we mean a $\bk$-linear rigid symmetric monoidal (with $\bk$-linear tensor product $\otimes$) category that is pseudo-abelian (additive and idempotent complete) and such that the endomorphism ring of the tensor unit $\unit$ is the base field $\bk$. A {\bf pseudo-tensor subcategory} of a pseudo-tensor category is a symmetric monoidal full subcategory closed under taking direct sums and direct summands. A {\bf thick tensor ideal} in a pseudo-tensor cateogry is a full subcategory closed under taking direct sums and direct summands and closed under tensor product with arbitrary objects. 

A pseudo-tensor category is a tensor category if it is abelian.

\subsection{Representations of finite groups}

Let $G$ be a finite group. Denote by $\Rep G=\Rep_{\bk}G$ the tensor category of finite-dimensional $G$-representations over $\bk$.

For $A,B\in \Rep G$, we take the convention that an {\bf extension of $A$ by $B$} has $A$ as a submodule and $B$ as a quotient.

We recall Definition~2.1 and Proposition~2.4, from \cite{CPW}:
\begin{definition}\label{DefWProj}
    For a fixed module $W\in \Rep G$, a module $M\in \Rep G$ is {\bf $W$-projective} if one of the following equivalent conditions is satisfied:
    \begin{enumerate}
        \item $M$ is a direct summand of $M'\otimes W$ for some $M'\in \Rep G$;
        \item $M$ is a direct summand of $M\otimes W\otimes W^\ast$;
        \item For every $W$-split epimorphism $M_1\tto M_2$ in $\Rep G$, the induced map
        $$\Hom_G(M,M_1)\to\Hom_G(M,M_2)$$
        is surjective.
    \end{enumerate}
\end{definition}

Here an epimorphism is {\bf $W$-split} if it becomes split after applying $W\otimes-$.

A representation in $\Rep G$ is {\bf algebraic} if there are only finitely many indecomposable summands (up to isomorphism) in its tensor powers. Equivalently, a representation is algebraic if its image in the split Grothendieck ring satisfies a polynomial equation.

\subsubsection{} Denote by $\Stab G$ the {\bf stable module category} of $G$, which is the quotient of $\Rep G$ by the (thick) tensor ideal of projective objects, see \cite{Benson-vector-bundles}*{\S 1.5}. This category is triangulated with shift functor $\Omega$. Concretely, $\Omega (M)$ is defined by a short exact sequence
$$0\to \Omega(M)\to P\to M\to 0$$
with $P$ projective. Dually, $\Omega^{-1}(M)$ is defined using an embedding of $M$ into its injective hull. Combinations of $\Omega$ and $\Omega^{-1}$ lead to well-defined operators $\Omega^i$ for all $i\in\mZ$ such that $\Omega^0(M)\cong M$ in the stable category. While $\Omega^i(M)$ is uniquely defined (up to isomorphism) in $\Stab G$, in $\Rep G$ it is defined only up to projective summands. In most cases where we refer to $\Omega(M)$ as an object in $\Rep G$, the ambiguity will not matter. In some cases we have to be more careful and unfortunately we will need to alternate between interpreting $\Omega(M)$ as a module without projective summands ($P$ as above is a projective cover) or consider $\Omega$ as the endofunctor of $\Rep G$ given by $K\otimes -$, where $K$ is the augmentation ideal in $\bk G$, which is thus a representative of $\Omega(\unit)$.

A module $M\in \Rep G$ is {\bf periodic} if $\Omega^tM\simeq\Omega^0 M$ (in $\Stab G$ or, equivalently, in $\Rep G$ up to projective summands) for some $t\in\mZ_{>0}$. The minimal such $t$ is the period of $M$.

\subsection{Elementary abelian \texorpdfstring{$p$}{p}-groups}
 We will use the notation regarding elementary abelian $p$-groups from \cite{Benson-vector-bundles}, and summarize some classical results from \cites{Carlson-var-coh-ring-mod,Carlson-coh-ring-mod,Benson-mod-rep-th,Benson-vector-bundles}. Take $r\in\mZ_{>0}$. The generators of $E=C_p^r$ are denoted by $g_i$, for $1\le i\le r$. We also set $$X_i=g_i-1\;\in\; \bk E.$$

 For $M\in \Rep E$, we define its support $\supp(M)\subset\mP^{r-1}(\bk)$ as the projective variety that has as cone the rank variety $V^{\mathrm{rk}}_E(M)\subset\mA^r(\bk)$ as in \cite{Benson-vector-bundles}*{\S 1.9}. Concretely, the point $[\mu_1:\mu_2:\cdots:\mu_{r}]\in \mP^{r-1}(\bk)$ is in $\supp(M)$ if and only if $M$ is not projective as a module over the algebra $\bk[t]/t^p$, where $t$ acts on $M$ as
 $$\sum_{i=1}^r\mu_i X_i\;\in\; \bk E.$$
Then for any $M,N\in\Rep E$,
\begin{equation}\label{sup:dual-shift}
\supp(M)\;=\;\supp(M^*)\;=\;\supp(\Omega(M)),
\end{equation}
\begin{equation}\label{sup:tensor}
\supp(M\otimes N)\;=\;\supp(M)\cap \supp(N),
\end{equation}
\begin{equation}\label{sup:oplus}
\supp(M\oplus N)\;=\;\supp(M)\cup \supp(N),
\end{equation}
and for any short exact sequence $M_1\to M_2\to M_3$ in $\Rep E$, 
\begin{equation}\label{sup:ses}
\supp(M_i)\;\subset\;\supp(M_j)\cup \supp(M_k) ,
\end{equation}
for any $i,j,k$ such that $\{i,j,k\}=\{1,2,3\}$. 

$M\in\Rep E$ is projective if and only if $\supp(M)=\emptyset$, and otherwise $M$ is periodic if and only if $\supp(M)$ is a point in $\mP^{r-1}$. Non-projective periodic modules in $\Rep E$ have period $1$ if $p=2$ and period at most $2$ if $p>2$ (and indecomposable non-projective periodic modules then have period $2$). The dimension of periodic modules is a multiple of $p^{r-1}$. Direct sums and direct summands of periodic modules are periodic.

\subsection{\texorpdfstring{$SL_2$}{SL₂}-representations}
We consider the tensor category $\Rep SL_2$ of rational representations of the affine group scheme $SL_2$ over $\bk$. The dominant weights are labelled by the natural numbers, and we denote by $T_i,\nabla_i,\Delta_i$ and $L_i$ the indecomposable tilting module, the costandard (Weyl) module, the standard module and the simple module with highest weight $i\in\mN$. The Weyl module can be realised as the degree $i$ polynomials in two variables, $\bk[x,y]_i$, where $SL_2$ acts on the span of $x,y$ via the tautological representation. We refer to \cite{Jantzen} for more details. The pseudo-tensor subcategory of $\Rep SL_2$ containing the tilting modules is denoted by $\Tilt SL_2$.

We will in particular work with the Steinberg modules
$$St_i\;=\; T_{p^i-1}\;=\;\nabla_{p^i-1}\;=\; L_{p^{i}-1}.$$
We will write in particular
$$St=St_1\quad\mbox{and}\quad S=St_{r-1},$$ when $r\in\mZ_{>1}$ is clear from context. Donkin's tensor product formula states that for $p-1\le i\le 2p-2$ and $j\in\mN$, we have
\begin{equation}\label{DTPT}T_{i+pj}\;\simeq\; T_i\otimes T_j^{(1)},\end{equation}
where $(\cdot)^{(k)}$ denotes the $k$-th Frobenius twist, for any $k\ge0$.

When we realise $SL_2$ as the group of $2\times 2$-matrices with determinant one, we denote by $T\simeq\mG_m$ the subgroup of diagonal matrices, $U\simeq \mG_a$ the group of matrices with diagonal entries $1$ and $0$ above the diagonal, and by $B=T\ltimes U$ the corresponding (negative) Borel subgroup.


\section{The conjecture} \label{sec:conjecture}
In this section, we state our main conjecture, and proceed to reformulate it in several interpretations. We also demonstrate why the conjecture becomes false if one tries to generalise it in some canonical directions. These results will not be used in the remainder of the paper, and we will actually apply some results from later sections to complete arguments in the current section.

\subsection{The first formulation}\label{SecFF}

Let $E$ be a finite subgroup of the additive group $\bk^+$. In particular $E\simeq C_p^r$, for some $r\in\mN$.

\subsubsection{} Via the subgroup inclusion 
$$E<\bk^+\simeq U(\bk)<SL_2(\bk),\quad \lambda\mapsto \left(\begin{array}{cc}
    1 &0  \\
    \lambda& 1 
\end{array}\right),$$
we have a restriction functor from the category of finite-dimensional algebraic representations $\Rep SL_2$ to $\Rep_{\bk} E$, and we denote by $\cT$ the essential image of the restriction functor
$$\Tilt SL_2 \to \Rep_{\bk} E.$$

\begin{lemma}[\cite{CF}*{Theorem~5.3.5, Remark~5.3.9(2)}, see also \cite{BE}*{Proposition~4.11}] \label{lem:restricted-tilting}
$\cT$ is already a pseudo-tensor subcategory. More precisely, the indecomposable modules in $\cT$ are, up to isomorphism, $T_i$, for $0\le i<p^r$, where we use the same notation for an $SL_2$-tilting module and for its restriction to $E$.  In particular $T_{p^r-1}\simeq \bk E$.
\end{lemma}

By \cite{Selecta}, we also have thick tensor ideals in $\cT$, 
$$0\subset \cT_{r+1}\subset \cT_r\subset\cT_{r-1}\subset\cdots\subset \cT_1=\cT ,$$
where $\cT_i$ has indecomposables $T_j$, for $p^{i-1}-1\le j$. For example, $\cT_{r+1}$ is simply the category of projective $\bk E$-modules. It follows immediately from the classification of tensor ideals in $\Tilt SL_2$ from \cite{Selecta} that these are the only thick tensor ideals in $\cT$.

\begin{aconjecture} \label{conj-a} Inside $\Rep_{\bk}E$, the subcategory $\cT_r$ is a thick tensor ideal. In other words, for every $X\in \Rep_{\bk}E$ and $T_i$ with $i\le p^{r-1}-1$, we have
\begin{equation}\label{eqXoT}
    X\otimes T_i\;\in\; \cT_{r}.
\end{equation}
\end{aconjecture}

One can define the full subcategory 
$$\widetilde{\cT}\;\subset\; \Rep_{\bk}E$$
on objects $X$ that satisfy \Cref{eqXoT}.

\begin{lemma} \label{lem:closure} $\widetilde\cT$ is an $\Omega$-invariant pseudo-tensor subcategory of $\Rep_\kk E$, i.e., it is closed under tensor products, duals, $\Omega$-shifts, direct sums, and taking direct summands.  
\end{lemma}

\begin{proof} Consider arbitrary $X,Y\in\widetilde\cT$, and $i\le p^{r-1}-1$.
As $T_i$ is self-dual, it is a direct summand in $T_i^{\o3}$. Thus $T_i\o(X\o Y)\stackrel\oplus\subset (T_i\o X)\o (T_i\o Y)\o T_i$ lies in $\cT_r$, due to the closure properties of the latter. This shows $X\o Y\in\widetilde\cT$.
The isomorphisms $T_i\o X^*\cong T_i^*\o X^*\cong (T_i\o X)^*$ show that $\widetilde\cT$ is closed under duals. 
As we have a stable isomorphism $T_i\o\Omega(X)\cong\Omega(T_i\o X)$ and $\cT_r$ contains all projective objects, $\Omega X\in\widetilde\cT$.
As $\cT_r$ is closed under direct sums and taking direct summands, $\widetilde\cT$ is closed under these operations, as well.
\end{proof}

Then \Cref{conj-a} simply states:
\begin{aconjecture} \label{conj-aa}
We have
  $\widetilde{\cT}=\Rep_{\bk} E$. 
\end{aconjecture}
    
\begin{remark} \label{rem:T-tilde}
    \begin{enumerate}
    \item For the moment, the only cases in which \Cref{conj-a} is proved are $r=1$ and $p^r=4$, see \cite{CF}. In this case, one can prove the conjecture directly from the classification of indecomposable modules and their fusion rules. In \Cref{sec:Loewy} we provide a new proof that does not rely on such explicit knowledge.
    \item The most relevant case is $E:= \mF_q^+$, for a finite subfield $\mF_q\subset \bk$ of order $q=p^r$, but thus far we have no indication that this case behaves differently from other cases.
    \item For the main application in \cite{CF}, in the formulation of \Cref{conj-aa} it would actually be sufficient to show that $\widetilde{\cT}$ contains the subcategory of $\Rep_{\bk} E$ consisting of direct summands of restrictions along $E\to S_{q}$ (as in \cite{CF}*{5.1.5--5.1.7}), with $q=p^r$, of representations in the subcategory $\mathbf{T}\mathbf{C}^{q}\subsetneq \Rep_{\bk} S_q$ from \cite{PolyF}*{Question~3.1.8}. 
    \item One can replace \eqref{eqXoT} with the ({\it a priori} weaker) condition $X\otimes T_i\in \cT$. Indeed, $T_i$ is a direct summand of $T_i\o T_i^*\o T_i$. So if $T_i\o X\in\cT$, then $T_i\o X\in \langle T_i\o T_i^*\o\cT\rangle_\oplus\subset\cT_r$, as $\cT_r$ is a thick tensor ideal in $\cT$
\end{enumerate}
\end{remark}

\begin{lemma}\label{lem:justS}
    For any $X\in\Rep E$, $X\in\widetilde{\cT}$ if and only if $S\otimes X\in\cT$, if and only if $S\otimes X\in\cT_r$.
\end{lemma}

\begin{proof}
By \cite{Selecta}, the thick tensor ideal $\cT_r$ is generated by
$$S:=St_{r-1}=T_{p^{r-1}-1}.
$$ This demonstrates the first equivalence, the second follows from \Cref{rem:T-tilde}(4).\end{proof}

\begin{lemma} \label{lem:restricted-sl-in} Let $X$ be any restriction to $E$ of an $SL_2$-module all of whose $SL_2$-composition factors have highest weights at most $p^r-1$. Then $X\in\widetilde\cT$.
\end{lemma}

\begin{proof} By \cite{BEO}*{Lemma~3.3} or \cite{AbEnv}*{Lemma~4.3.4}, $S\o L_i$ is a tilting $SL_2$-module for all $0\le i\le p^r-1$. As tilting modules have trivial extensions, this means $S\o X\in\cT$. Hence we can apply \Cref{lem:justS}.
\end{proof}

\begin{lemma} \label{lem:SL2-std} The restriction of any standard or costandard $SL_2$-module $X$ lies in $\widetilde\cT$.
\end{lemma}

\begin{proof} By duality, it suffices to consider $X=\nabla_i$, the restriction to $E$ of the costandard module with highest weight $i$, for $i\ge0$. 

We can think of $\nabla_i$ as the span of all monomials of degree $i$ in $\kk[x,y]$. It is a $U(\kk)$-module, and hence an $E$-module, where the action of a lower unitriangular matrix with off-diagonal entry $\lambda$ sends $x^a y^b$ to $(x+\lambda y)^a y^b$. For each $k\ge0$, there is a short exact sequence of vector spaces
$$
0\to
\nabla_{k-1}
\xrightarrow{f}
\nabla_{k+i}
\xrightarrow{g}
\nabla_i
\to0 ,
$$
where $f(x^a y^b)= x^a y^{b+i+1}$ and $g(x^a y^b)$ is $x^{a-k} y^b$ if $b\le i$, and zero otherwise. $f$ is always a $U(\kk)$-module morphism, and one can check that $g$ is one if $p$ divides $k$. Thus, in this case, the above is in particular a short exact sequence of $E$-modules.

We can now specialize $k$ to $p^r$, then $\nabla_{p^r-1}=T_{p^r-1}$ is projective, and hence injective, over $E$ by \Cref{lem:restricted-tilting}, so $\nabla_{p^r+i}$ and $\nabla_i$ agree up to projective summands, for all $i$, and it suffices to consider $X=\nabla_i$ for $i<p^r$. But then $X\in\widetilde\cT$ by \Cref{lem:restricted-sl-in}.
\end{proof}

The following lemma is not used in the rest of the paper, but relies on later results.

\begin{lemma} \label{lem:SL2-simples} Let $X\in\Rep E$ be any Frobenius twist of a restricted $SL_2$-tilting module, or any restricted simple $SL_2$-module. Then $X\in\widetilde\cT$.
\end{lemma}

\begin{proof} By Steinberg's tensor product theorem, the second assertion follows from the first assertion, using \Cref{lem:closure}. Now it suffices to consider $X=L_1^{(j)}$, for $j\ge0$, as any Frobenius twist of a tilting module is a direct summand in a tensor power of the latter. But such $X$ are $2$-dimensional and of Loewy length two, so they are in $\widetilde\cT$ by \Cref{Exam:1} or \Cref{cor:ll-2-any-r} below.
\end{proof}

The fact that $S$ is a generator of $\cT_r$ also means that \Cref{conj-a} is equivalent to the following restricted version:

\begin{aconjecture} \label{conj-b}
    For every $X\in \Rep_{\bk}E$,  we have $X\otimes S \in \cT_r$.
\end{aconjecture}

Moreover, \Cref{conj-b} is then equivalent to the following formulation.

\begin{aconjecture} \label{conj-c}
    The category of $S$-projective modules in $\Rep E$ is precisely $\cT_r$.
\end{aconjecture}

Before we move on to further reformulations of the conjecture, we formulate some statements that are implied by the conjecture. For the moment, we have no proof for any of them.

For reference below, observe that restricting the $SL_2$-representation to $B=U\rtimes T$ equips
$$\bk C_p^r\;\simeq\;\End_U(St_r),$$
see \Cref{LemmaLambda}, with the structure of an algebra in the category of $T$-representations. In other words, the group algebra $\bk C_p^r$ inherits a $\mZ$-grading. 

\begin{lemma}\label{lem:consequences} If \Cref{conj-a} is true, then the following statements are true for every $S$-projective module $Y$ in $\Rep_\kk E$:
    \begin{enumerate}
        \item $Y$ is self-dual.
        \item $Y$ has a graded lift, with respect to the $\mZ$-grading of $\bk E$ above.
        \item $Y$ has a filtration with every subquotient a copy of $S$.
        \item $Y$ is the restriction of an $SL_2$-module.
        \item Assume $Y$ is indecomposable and not projective, set $t:=\Hom_E(\unit, Y)$ and $d:=\dim_\kk Y$. Then $t=2^\ell$ and $d=\ell' 2^\ell p^{r-1}$ for some $0\le\ell\le r-1$ and $1\le\ell'\le p-1$. In particular, $t\le 2^{r-1}$, $d$ is divisible by $p^{r-1}$, and $d\le (p-1)(2p)^{r-1}$.
    \end{enumerate}
\end{lemma}
\begin{proof}
Part (1) is immediate since all tilting modules of $SL_2$ are self-dual already over $SL_2$.

For part (2), consider the algebra homomorphisms
$$\bk E\;\hookrightarrow\;\bk U(\bk)\;\to\; \End_U(St_r).$$
These are algebra homomorphisms in the category of $T(\bk)$-representations, such that the composite is an isomorphism and the source and target are rational $T$-representations. In other words, the inclusion of $\bk E$ into $\bk U(\bk)$ is one of algebras with $T(\bk)$-actions. Now an arbitrary $X\in \Rep SL_2$ has a (weight) $\mZ$-grading, and since the action of $E$ comes from
$$\bk E\;\hookrightarrow\;\bk U(\bk)\;\to\; \End_{\bk}(X),$$
it follows that this grading is compatible with the $\bk E$-module structure. Applying this to tilting modules proves (2).

Now we prove part (3). It follows from \Cref{lem:self-ext1} below that $T_{j}$ for $p-1\le j\le 2p-2$ is a self-extension of $St_1$ over $E$. It then follows from iteration of Donkin's tensor product formula \eqref{DTPT}, that every $T_i$ for $i\ge p^{r-1}$ is of the form
$$T\otimes X^{(r-1)}$$
where $T$ is a tilting module which has, as an $E$-representation, a filtration with quotients $S$, and $X$ is some tilting module. Taking a Loewy filtration of $X$ as an $E$-representation then concludes the proof. Note that for $p=2$, we do not need \Cref{lem:self-ext1}
as then the result follows almost immediately from Donkin's tensor product formula.

Part (4) is immediate. 

For part (5), we may assume $Y\cong T_i$ for some $p^{r-1}-1\le i\le p^r-2$. We can write $i+1=\sum_{k=0}^{r-1} i_k p^k$ with $i_{r-1}\neq0$. Set $\ell':=i_{r-1}$ and let $\ell$ be the number of non-zero non-leading coefficients in the above $p$-adic expansion of $i+1$.

To describe the possible values of $t$, we can use that by \Cref{PropT}, $t=\ell^\nabla(T_i)$, the length of the $\nabla$-flag of $T_i$ as an $SL_2$-module. The $\nabla$-modules appearing in the $\nabla$-flag of $T_i$ are given by $\nabla_{i'-1}$, where $i'$ ranges over all numbers obtained by inverting the signs of an arbitrary subset of the non-leading coefficients $\{i_0,\dots,i_{r-2}\}$ in the $p$-adic expansion of of $i+1$ by \cite{STWZ}*{Proposition~3.3}. There are $2^\ell$ distinct numbers $i'$ obtained in this way.

To describe $d$, we use that
$$
T_i \cong T_{p-1+i_0}^{(0)}\o T_{p-1+i_1}^{(1)}\o\dots\o T_{p-1+i_{r-2}}^{(r-2)}\o T_{i_{r-1}-1}^{(r-1)}
$$
by \cite{STWZ}*{Proposition~3.4}, where the superscripts indicate Frobenius twists. So $\dim T_i = \ell' 2^\ell p^{r-1}$. This number is always divisible by $p^{r-1}$, and is at most $(p-1) (2p)^{r-1}$. Note that this maximal dimension is attained for $T_{p^r-2}$. 
\end{proof}

Note that \Cref{lem:consequences}(5) is compatible with what we show in \Cref{lem:dim-periodic} below.

\subsection{The second formulation}
Fix $r\in\mZ_{>0}$.
Let $V$ be a faithful two-dimensional $C_p^r$-representation.

\begin{aconjecture} \label{conj-d}
Let $\cT$ be the pseudo-tensor subcategory generated by $V$ in $\Rep_{\bk}C_p^r$. Let $\cP\subset\cT$ be the category of projective modules in $\Rep_{\bk}C_p^r$. The unique
minimal thick tensor ideal in $\cT$ among those that strictly contain $\cP$ remains a thick tensor ideal in $\Rep_{\bk}C_p^r$.
\end{aconjecture}

To argue that \Cref{conj-d} is equivalent to the previous three, and for later use, we first introduce some notation.

\subsubsection{}\label{Sec2dim}Any indecomposable two-dimensional $C_p^r$-representation has a basis $v_1,v_2$ so that
$$X_iv_1\;=\;\lambda_i v_2,\quad\mbox{and}\qquad X_iv_2=0,\quad\mbox{for some }\quad
\ulambda:=(\lambda_1,\ldots,\lambda_r)\;\in\; \mA^r(\bk)\backslash \{0\}$$
and denote the module by $V_{\ulambda}$.
The modules $V_{\ulambda}$ and $V_{\umu}$ are isomorphic if and only if $\ulambda$ and $\umu$ have the same image in $\mP^{r-1}(\bk)$. We will alternate between interpreting $\ulambda$ in $\mA^r(\bk)\backslash\{0\}$ and $\mP^{r-1}(\bk)$. We also write
$$\ulambda^{(i)}:=(\lambda_1^{p^i},\ldots,\lambda_r^{p^i})$$
for $i\in\mN$, and hence $(V_{\ulambda})^{(i)}\simeq V_{\ulambda^{(i)}}$, with the former the $i$-th Frobenius twist of $V_{\ulambda}$.
If we interpret $\ulambda\in\mA^r$, then $V_{\ulambda}$ has a canonical short exact sequence
$$0\to\bk\xrightarrow{1\mapsto v_2} V_{\ulambda}\to \bk\to 0.$$
The representation $V_{\ulambda}$ defines a group homomorphism
$$C_p^r\;\to\; SL_2(\bk),\quad g_i\mapsto \left(\begin{array}{cc}
    1 &0  \\
    \lambda_i & 1 
\end{array}\right)$$
so that $V_{\ulambda}$ is the restriction of the vector representation $V=T_1$.

\begin{lemma}\label{Lem:Moore}
    The following conditions are equivalent on 
    $\ulambda=[\lambda_1:\lambda_2:\cdots:\lambda_r]$ in $\mP^{r-1}(\bk)$:
    \begin{enumerate}
        \item $V_{\ulambda}$ is faithful;
        \item The elements $\lambda_i\in\bk$ are linearly independent over $\mF_p$.
        \item The determinant of the Moore matrix 
        $$ \mathbb{M}(\ulambda):=\begin{pmatrix}
\lambda_1 & \lambda_2 & \lambda_3 & \cdots & \lambda_r  \\
\lambda_1^p & \lambda_2^p & \dots \\
\dots \\
\lambda^{p^{r-1}}_1 & \dots&&&\lambda_r^{p^{r-1}}
\end{pmatrix}$$
is non-zero.
    \end{enumerate}
\end{lemma}
\begin{proof}
    Equivalence of (1) and (2) follows from the group homomorphism before the lemma. Equivalence between (2) and (3) is \cite{Go}*{Lemma~1.3.3}.
\end{proof}
We denote the set of $\ulambda$ satisfying the equivalent properties in \Cref{Lem:Moore} by $\mA_f^r$ or $\mP_f^{r-1}$.

By identifying $C_p^r$ with its image in $SL_2$, the unique ideal in \Cref{conj-d} corresponds to $\cT_r$, by the discussion leading up to \Cref{conj-a}. Hence \Cref{conj-d} is just a direct reformulation of \Cref{conj-a}.

\subsection{The third formulation}

\subsubsection{}
Let $G$ be a finite group and take a faithful $X\in \Rep_{\bk} G$. We consider the full subcategory $\cT_X$ of $\Rep_{\bk}G$ comprising direct sums of direct summands in $\Omega^m(X^{\otimes n})$, $m\in \mZ$ and $n\in\mZ_{>0}$. We denote by $\cI_X$ the thick tensor ideal in $\Rep_{\bk}G$ generated by $X$. Then we have $\cT_X\subset\cI_X$.
One can ask whether $\cI_X=\cT_X$. Naturally, this will very rarely be the case.

\begin{aconjecture} \label{conj-e}
    For $G=C_p^r$, with $r\in\mZ_{>1}$ and $X=St_{r-1}$, via some embedding $C_p^r<\bk^+<SL_2(\bk)$ as in \Cref{SecFF}, we have $\cI_X= \cT_X$.
\end{aconjecture}

To prove equivalence with the previous formulations, more specifically \Cref{conj-b}, it suffices to show that $\cT_{St_{r-1}}$ is $\cT_r$. This is precisely the content of \Cref{Lem:TOT} below.

One can further reformulate (and try to understand better) the conjecture, by obtaining more conceptual definitions for the modules in $\Rep_{\bk} C_p^r$ `of the form' $St_{r-1}$ for embeddings into $SL_2$. For example, when $r=2$, we have:

\medskip

\begin{aconjecture} \label{conj-e2}
    For $G=C_p^2$, and $X$ a self-dual periodic faithful module of dimension $p$, we have $\cI_X= \cT_X$.
\end{aconjecture}

Indeed, we will prove that $X$ as in \Cref{conj-e2} are precisely restrictions of Steinberg modules, for $r=2$ in \Cref{lem:uniqueselfdual}. Note that this characterisation fails for $r>2$, see \Cref{Sec8}.

\begin{remark}
   
        We will show in \Cref{Sec9} that \Cref{conj-e2} fails if we remove the condition that $X$ be self-dual. Similarly, we show in \Cref{Sec8} that \Cref{conj-e} fails (for $r>2$) if we generalise it to all $X\in \Rep C_p^r$ that are, like the Steinberg modules $St_{r-1}$, self-dual periodic faithful and of dimension $p^{r-1}$.
        
 We regard this ease with which we can disprove the property $\cI_X=\cT_X$ for lookalikes $X$ of $St_{r-1}$ as further evidence for the validity of the conjecture.

\end{remark}

\subsection{Some context for \texorpdfstring{$q=9$}{q=9}}\label{Sec9}
Here we consider $p=3$ and $E=C_3^2$, the elementary abelian $3$-group of order $q=9$ (\textit{i.e.} $r=2$).

For $\ulambda\in \mP^1_f$, we can write
$\ulambda=[1:\lambda]$, for a unique $\lambda\in \bk\backslash \mF_3$.
We set
$$\tilde{\lambda}:=\binom{\lambda}{2}\;=\;\frac{\lambda(\lambda-1)}{2}\;=\;\lambda-\lambda^2\;\in\bk,$$
and consider the matrices in $\Mat_3(\bk)$
$$\mI= \left(\begin{array}{ccc}
       1  & 0 &0 \\
       0  & 1&0\\
       0&0&1
    \end{array}\right)\quad\mbox{and}\quad \mJ=\left(\begin{array}{ccc}
       0  & 0 &0 \\
       1  & 0&0\\
       0&1&0
    \end{array}\right).$$ We interpret $\mJ^0$ as $\mI$.

    For this value of $\ulambda$, the module $S=St_1$ for $E$ is given by
    $$X_1\mapsto \mJ,\quad X_2\mapsto \lambda\mJ+\tilde{\lambda} \mJ^2$$
More generally, we can parametrise the three-dimensional periodic modules with the same support $\{\blambda\}$, with $\blambda=[\lambda:-1]$. These modules are $S(\mu)$, for $\mu\in\bk$, where $S(\mu)$ corresponds to
$$X_1\mapsto \mJ,\quad X_2\mapsto \lambda\mJ+(\tilde{\lambda}+\mu) \mJ^2.$$
In particular, $S(0)=St_1$ and $S(\mu)=M_{\lambda,\tilde{\lambda}+\mu}$ in the notation of \cite{Benson-vector-bundles}*{Example~1.13.1}, which furthermore shows that
$$S(\mu)^\ast\simeq S(-\mu).$$

We investigate six-dimensional periodic modules.

\begin{lemma}\label{LemClas3}
    \begin{enumerate}
        
        \item The isomorphism classes of indecomposable six-dimensional $E$-representations with support $\{\blambda\}$ are labelled by $\bk\times\{0,1,2\}$. This assignment sends $(\mu,i)\in\bk\times \{0,1,2\}$ to the $E$-representation $N^{(i)}(\mu)$ with
        $$g_1\mapsto \left(\begin{array}{c|c} \mI+\mJ &0\\\hline
        0&\mI+\mJ\end{array}\right) \quad\mbox{and}\quad g_2\mapsto\left(\begin{array}{c|c} \mI+\lambda\mJ+(\tilde{\lambda}+\mu)\mJ^2 &0\\\hline
        \mJ^i&\mI+\lambda\mJ+(\tilde{\lambda}+\mu)\mJ^2\end{array}\right).$$
        \item The modules $N^{(0)}(\mu)$ are cyclic. The cyclic submodules of maximal dimension in $N^{(1)}(\mu)$, resp. $N^{(2)}(\mu)$, have dimension five, resp. four.
        \item The only non-split first extensions between the modules $S(\mu)$ take the form
        \begin{enumerate}
            \item $0\to S(\mu)\to N^{(0)}(\kappa)\to S(\nu)\to 0$, for $2\kappa=\mu+\nu$;
            \item $0\to S(\mu)\to N^{(1)}(\kappa)\to S(\nu)\to 0$, for $2\kappa=\mu+\nu$;
            \item $0\to S(\mu)\to N^{(2)}(\mu)\to S(\mu)\to 0$.
        \end{enumerate}
    \end{enumerate}
\end{lemma}
\begin{proof}
Let $N$ be a 6-dimensional $E$-representation with support $\{\blambda\}$. In particular, $N$ is projective over $C_3\times 1$ and we can take a basis where $g_1$ acts as in part (1) and $g_2$ via block matrix with zero in the top right block. In other words, $N$ is an extension of two 3-dimensional representations that are projective over $C_3\times 1$. 
Hence, the support of each of these $3$-dimensional representations is a proper subvariety of $\mP^1$, i.e., a finite, possibly empty set of points. Being of dimension $3$, the representations cannot be projective over $E$, so their support is a non-empty finite set of points. As the support variety of an indecomposable representation is connected by \cite{Carlson-connected}, each of these representations is a direct sum of representations whose support is a point. The latter summands are periodic, and of dimension at least $3$. Hence, the support of each of the $3$-dimensional representations is a point. Now by \Cref{sup:ses}, both of these points have to be $\blambda$. So the $3$-dimensional modules are of the form $S(\mu)$ and $S(\nu)$. 

To conclude the proof, we thus need to analyse the extensions between modules of the form $S(\mu)$. To simplify expressions one can apply an automorphism $\phi$ of $\bk E$ which sends $X_1$ to $X_1$ and $X_2$ to $X_2+\lambda X_1+\tilde{\lambda}X_1^2$. Then $S(\mu)$ is the twist by $\phi$ of the module $R(\mu)$ where $X_1$ acts by $\mJ$ and $X_2$ acts by $\mu\mJ^2$. One thus needs to understand
$$\Ext^1(R(\mu),R(\nu))\;\simeq\;\uHom(\Omega R(\mu), R(\nu)).$$
Here, $\Hom(\Omega R(\mu), R(\nu))$ always has dimension 3 and is a cyclic $\End(R(\nu))$-module. The dimension of $\uHom(\Omega R(\mu), R(\nu))$ is 3 when $\mu=\nu$ and 2 otherwise, due to the non-zero composite
$$\Omega R(\mu)\hookrightarrow \bk E\tto R(\nu).$$
To conclude the lemma one then needs to show that isomorphisms between the middle terms in (3)(a) respectively (3)(b) are governed by $\mu+\nu$. That the middle terms only depend on $\mu+\nu$ follows from a direct base change. That $N^{(i)}(\kappa)$ is not isomorphic to $N^{(i)}(\kappa')$ for $\kappa\not=\kappa'$ follows by directly verifying that for all submodules $S(\kappa)$ in $N^{(i)}(\kappa)$, the quotient is $S(\kappa)$.
\end{proof}

\begin{proposition}\label{Prop3d}
\begin{enumerate}
    \item  If $\mu+\nu\not=0$, we have
    $$
    S(\mu)\otimes S(\nu)= S\left(\frac{\mu\nu}{\mu+\nu}\right)\oplus N^{(0)}\left(\frac{\mu\nu}{\mu+\nu}\right).
    $$
    \item  If $\mu=0=\nu$, then
    $$
    S(0)\otimes S(0)= S\left(0\right)\oplus N^{(1)}\left(0\right).
    $$
    \item Finally, if $\mu\not=0$ and $\mu+\nu=0$, then $S(\mu)\otimes S(\nu)$ is an indecomposable non-projective module $Q$, independent of $\mu$, of dimension $9$ with $\dim\End(Q)=15$ and a Loewy series with dimensions $3,3,3$.
\end{enumerate}
   
\end{proposition}

\begin{proof}
    This follows from straightforward but tedious computations. For example, assume that $\mu+\nu\not=0$.

    We label the bases of $S(\mu)$ and $S(\nu)$ corresponding to the above matrix definitions by $\{e_1,e_2,e_3\}$ and $\{f_1,f_2,f_3\}$. Then we can verify that $e_1\otimes f_1$ generates a (cyclic) submodule $R$ which fits into a short exact sequence
    $$0\to S(-\mu\nu/(\mu+\nu))\to R\to S(0)\to 0,$$
    where the submodule is (as is clear from the claimed short exact sequence and cyclicity of $R$) generated by 
    $$X_2(e_1\otimes f_1)-\lambda X_1(e_1\otimes f_1)-\tilde{\lambda}X_1^2(e_1\otimes f_1)$$
    By \Cref{LemClas3}, $R$ is then isomorphic to $N^{(0)}(\mu\nu/(\mu+\nu))$.

    Finally, we can find a complementary submodule isomorphic to $S(\mu\nu/(\mu+\nu))$ generated by 
    $$e_1\otimes f_2-e_2\otimes f_1-(e_1\otimes f_3-e_3\otimes f_1)+\alpha X_1(e_1\otimes f_1)+\beta X_1^2(e_1\otimes f_1),$$
    for appropriate $\alpha,\beta\in \bk$, most importantly $\alpha=(\nu-\mu)/(\mu+\nu)$.
\end{proof}

\begin{remark}
The special case $\mu=\nu$ of \Cref{Prop3d} is easily verifiable. Indeed, in this case 
$$\wedge^2 S(\mu)\;\simeq\;S(\mu)^\ast\;\simeq\;S(-\mu).$$
 In this case $\alpha=0=\beta$ in the proof of \Cref{Prop3d} and the short exact sequence is simply
$$0\to S(\mu)\to R\to S(0)\to 0.$$
\end{remark}

If we consider the group homomorphism $C_3^2\to SL_2(\bk)$ corresponding to $\ulambda$, then
    $$S=St_1=T_2\mapsto S(0),\quad T_3\mapsto N^{(2)}(0),\quad T_4\mapsto N^{(1)}(0),\quad T_5\mapsto N^{(0)}(0). $$
This leads to:

\begin{example}
   For all $\nu\neq0$, we have
    $$S\otimes S(\nu)\;\simeq\; S \oplus T_5,$$
    in line with \Cref{conj-b}.
\end{example}

In sharp contrast, we find:

\begin{corollary}\label{Cor:Anti}Assume that $\bk$ is uncountable.
    Whenever $\mu\not=0$, we have $\cI_{S(\mu)}\not\subset \cT_{S(\mu)}$. Hence $S(\mu)$ `does not satisfy~\Cref{conj-e}'.
\end{corollary}
\begin{proof}
    By construction, the number of indecomposables in $\cT_{S(\mu)}$ is countable. However, by \Cref{Prop3d}, the number of indecomposables in $\cI_{S(\mu)}$ is at least the order of the field.
\end{proof}

\begin{remark}
    The assumption that $\bk$ be uncountable in \Cref{Cor:Anti} is only for convenience, since the module $S(\mu)$ is actually algebraic.
\end{remark}



\begin{remark}
    \Cref{Prop3d}(1) also demonstrates that for $S(\mu)$ with $\mu\not=0$, indecomposable $S(\mu)$-projective modules need not have a filtration with all subquotients isomorphic to $S(\mu)$, in contrast to \Cref{lem:consequences}(3).
\end{remark}

\subsection{Some context for \texorpdfstring{$q=8$}{q=8}}\label{Sec8}
Here we consider $p=2$ and $E$ an elementary abelian $2$-group of order $q=8$ (\textit{i.e.} $r=3$).

The most immediate properties of $S=St_2=T_{7}$ are that it is $4$-dimensional and periodic, self-dual and faithful as a $E=C_2^3$-representation. We study all such modules, and observe that self-duality is automatic.
\begin{lemma}\label{Lem:4dim}
Let $M$ be a faithful periodic 4-dimensional module in $\Rep E$.
\begin{enumerate}
    \item $M\simeq A\otimes B$, for two-dimensional modules $A,B$.
    \item $M$ is self-dual.
    \item $M$ is algebraic.
\end{enumerate}
\end{lemma}
\begin{proof}
    Since two-dimensional modules are self-dual and algebraic, see \cite{CF}*{Remark~5.3.9(3)}, we only need to prove part (1).

    Consider therefore $[1:\mu:0]$ and $[1:\alpha:\beta]$ in $\mP^2$. Then, we can directly compute that $A\otimes B$ is the $E$-representation
    $$\bk E/(X_1+aX_2+bX_3+cX_2X_3),$$
    with 
    $$\mu=a^{-1},\quad \beta=\frac{1+a}{b+c},\quad \alpha=\frac{b+ca^{-1}}{b+c}.$$
Up to relabelling the $X_i$, every faithful periodic 4-dimensional module is of the above form.
\end{proof}


    Now we single out which such modules correspond to Steinberg modules. For a faithful periodic 4-dimensional module $M$, call a 2-dimensional module $W$ `complementary' if $M\otimes W$ is projective. Any such module defines a $\bk$-linear morphism
    $$\beta_W:\;H^1(E,\bk)\;\to\;\End(M),$$
by sending a self-extension 
$$0\to \unit \xrightarrow{f} Y\xrightarrow{g}\unit\to 0$$
to the left vertical arrow in the commutative diagram (the middle vertical arrow is not unique)
$$\xymatrix{
0\ar[r]&M\ar[r]\ar[d] &M\otimes W\ar[r]\ar[d]&M\ar[r]\ar@{=}[d]&0\\
0\ar[r]&M\ar[r]^-{M\otimes f} &M\otimes Y\ar[r]^-{M\otimes g}& M\ar[r]&0
}$$
where the top row is defined similarly to the bottom row. Note that the left vertical arrow is uniquely defined as a consequence of the equality $\uEnd(M)=\End(M)$. The latter equality follows for instance from dimensional arguments. Indeed, $M\otimes M^\ast$ cannot contain an a projective direct summand, since, by support considerations, its complement $Q$ would be projective over some subgroup of $E$ isomorphic to $C_2^2$ which leads to a contradiction (using \Cref{Lem:4dim}(1))
$$
4\le\dim\End(M) = \dim\Hom(\one,M\o M^*) \le 1+(\dim_\kk Q)/4 = 3.
$$

    \begin{proposition}\label{Prop:4dim}
        The following conditions are equivalent on 4-dimensional faithful periodic modules $M$ over $E=C_2^3:$
    \begin{enumerate}
        \item $M$ is the restriction of $St_2$ for some embedding of $E$ into $SL_2$;
        \item $M\simeq V\otimes V^{(1)}$ for a (two-dimensional) module $V$;
        \item There exists a complementary $W$ for which the image of
        $$\beta_W:\,H^1(E,\bk)\to\End(M)$$
        is a subalgebra.
        \item There exists a complementary $W$ for which the image of 
        $$H^1(E,\bk)\xrightarrow{\beta_W}\End(M)\tto\End(M)/\Aut(M)$$
        is finite.
    \item There are finitely many isomorphism classes of modules of the form $M\otimes Y$, with $Y$ a two-dimensional module.
    \end{enumerate}
    \end{proposition}
\begin{proof}
    Equivalence between (1) and (2) is straightforward. We write $M=A\otimes B$ (with $A$ and $B$ non-isomorphic two-dimensional modules) and observe
    $$\End(M)\;\simeq\;\End(A)\otimes \End(B)\;\simeq\; k[t]/t^2\otimes k[s]/s^2.$$
    It follows that $\End(M)/\Aut(M)$ consists of two points (represented by an arbitrary isomorphism and an arbitrary morphism with 1-dimensional image) and one $\mP^1$-family represented by morphisms with 2-dimensional image. A 3-dimensional subspace of $\End(M)$ which contains $1$ is a subalgebra if and only if its image in $\End(M)/\Aut(M)$ is finite. For these modules we can also observe that 
    $$\End(M)/\Aut(M)\;\simeq\; \Ext^1(M,M)/\Aut(M)$$ is canonically in bijection with the isomorphism classes of self-extensions of $M$. Here, the isomorphism between $\End(M)$ and $\Ext^1(M,M)$ can be realised such that $\beta_W$ factors as
    $$H^1(E,\bk)\xrightarrow{\sim}\Ext^1(\unit,\unit)\xrightarrow{M\otimes-}\Ext^1(M,M)\xrightarrow{\sim}\End(M).$$
    It follows that (3) and (4) and (5) are equivalent.

    That (2) implies (3) will be proved rigorously in the proof of \Cref{ThmCarlson}. To show that (3) implies (2), we can observe that, under the above identifications, the dimension of the image of an endomorphism of $M$ corresponds (by adding 4) to the dimension of the submodule of $A\otimes B\otimes Y$ generated by the tensor product of the three generators. Since a 3-dimensional subspace of $\End(M)$ is a subalgebra if and only if it contains an element with image of dimension 1, it suffices to prove that the tensor product of generators in $A\otimes B\otimes Y$ can only be of dimension 5 when (up to permutation) $A=B^{(1)}$ and $Y=B^{(1)}$, which follows from a direct computation.
 \end{proof}

    \begin{corollary}
Let $X$ be a self-dual faithful periodic $\bk E$-module of dimension 4 that is not of the form $St_2$, for any embedding $E<SL_2(\bk)$, then $\cI_X\not=\cT_X$.
    \end{corollary}
    \begin{proof}
        By \Cref{Lem:4dim} $\cT_X$ contains finitely many indecomposable objects, while $\cI_X$ contains infinitely many by \Cref{Prop:4dim}.
    \end{proof}


\section{Analysing the restrictions of Steinberg modules} \label{sec:analyse}

We continue the set-up from the previous section.
We thus have an inclusion of groups $E:=C_p^r\hookrightarrow \bk^+$, which we further interpret as an inclusion into $SL_2(\bk)$. This information is equivalent to the data of some $$\ulambda=(\lambda_1,\ldots,\lambda_r)\in \mA^r_f(\bk).$$

Denote by $T_i$, $0\le i< p^r$  the image of the corresponding $SL_2$-tilting module in $\Rep E$.

\subsection{Endomorphisms of Steinberg modules}
The following was proved in \cite{CF}*{Proposition~ 5.3.8}, where $\ell^\nabla$ refers to the length of the $\nabla$-flag of a module.
\begin{proposition}\label{PropT}
For all $0\le i<p^r$, we have 
    \begin{enumerate}
    \item
    $\dim_{\bk}\Hom_{E}(\unit,T_i)\,=\,\ell^\nabla(T_i),$
    \item $\dim_{\bk}\Hom_{E}(V,T_i)\,=\,2\ell^\nabla(T_i)-\dim_{\bk}\Hom_{SL_2}(\unit,T_{i})$.
    \end{enumerate}
\end{proposition}

\begin{lemma}\label{LemCyc}
The $C_p^r$-modules $St_i$ are cyclic for $0\le i\le $r.
\end{lemma}
\begin{proof}
    Since $St_i=\nabla_{p^i-1}$, and the module is self-dual, this follows from \Cref{PropT}(1).
\end{proof}

\subsubsection{}\label{basis}
By abuse of notation we choose a basis written as $\{e_1,e_2,\ldots, e_p\}$ of $St_1^{(j)}$ for each $0\le j <r$. We choose the basis of $St_1$ via the following isomorphism
$$ \bk[x,y]_{p-1}\,\xrightarrow{\sim}\, St_1,\qquad x^{p-i}y^{i-1}\mapsto (p-i)! \,e_i,\quad\tfor 1\le i\le p.
$$
Similarly, the basis of $St_1^{(j)}$ is chosen such that $e_i$ in $St_1^{(j)}$ corresponds to $e_i^{\otimes p}$ in the realisation of $St_1^{(j)}$ as a subquotient of $(St_1^{(j-1)})^{\otimes p}$.

Due to this simplification of notation it is important to take into account the convention that 
$$\bigotimes_i St_1^{(a_i)}\;:=\; St_1^{(a_1)}\otimes St_1^{(a_2)}\otimes \cdots,\quad \text{for }\{a_i\in\mN\}.$$

\subsubsection{}\label{EndS1} With our choice of basis of $St_1$, one can now verify that
$$\End_U(St_1)\;\simeq\;\End_E(St_1)\;\simeq\; \bk[z]/z^p,$$
where $z^a$, for $0\le a<p$, corresponds to the endomorphism
$$e_i\;\mapsto \; \begin{cases} e_{i+a}&\text{ if }i+a<p\\
0&\text{ otherwise.}
\end{cases}$$
Indeed, by \Cref{LemCyc}, we know that $St_1$ is cyclic over $E$, so the dimension of the endomorphism algebra is bounded by $p$. It thus suffices to observe that each of the suggested endomorphisms is $U$-linear.

\begin{lemma}\label{LemmaLambda}
    There is an algebra isomorphism
    $$\bk[z_1,\ldots, z_{r-1}]/(z_i^p\mid 1\le i<r)\;\xrightarrow{\sim}\; \End_E(S)=\uEnd_E(S),$$
    where each basis element $\prod_{j}z_j^{a_j}$,
 is sent to the endomorphism determined by
    $$e_1\otimes e_1\otimes\cdots \otimes e_1\;\mapsto \; e_{1+a_1}\otimes e_{1+a_2}\otimes \cdots \otimes e_{1+a_{r-1}}.$$
    
\end{lemma}
\begin{proof}
    We have the obvious injective algebra morphism
    $$\left(\bk[z]/(z^p)\right)^{\otimes r-1}\simeq\bigotimes_{i=0}^{r-2}\End_E(St_1^{(i)})\;\hookrightarrow\;\End_E(S).$$
    Moreover, since $S$ is cyclic by Lemm\Cref{LemCyc}, we have
    $$\dim_k\End_E(S)\;\le\; \dim_k S \;=\; p^{r-1},$$
    so the above injection must be a bijection.

    It only remains to show that $\End_E(S)$ equals its quotient $\uEnd_E(S)$. This follows from adjunction, since 
    $S^\ast\otimes S=S^{\otimes 2}$ does not contain a projective summand. Indeed, $S^{\otimes 2}$ is a direct sum of $T_j$ with $j\le 2p^{r-1}-2<p^r-1$, so we can use \cite{CF}*{Remark~5.3.9(2)}.
\end{proof}

\subsection{Support and Heller shifts of tilting modules}\label{SecSuppHell}

\begin{proposition} \label{prop:Omega-Ti}
    In $\Stab E$ we have
    $$\Omega( T_{a+p^{r-1}b})\;\simeq\; T_{a+ p^{r-1}(p-2-b)},$$
    for all $p^{r-1}-1\le a\le 2p^{r-1}-2$ and $0\le b\le p-2$.
\end{proposition}
\begin{proof}
    By iterating Donkin's formula, \cite{Jantzen}*{E.9}, we have
    $$T_{a+p^{r-1}b}\;\simeq\; T_{a}\otimes T_b^{(r-1)}$$
    in the chosen range of $a,b$. Over $U$, so certainly over $E<U(\bk)$, we have a short exact sequence
    \begin{equation}\label{eqsesT}0\to T_{p-2-b}\to T_{p-1}\to T_b\to 0,\end{equation}
    where all tilting modules are also (co)standard objects. Also by Donkin's formula, 
    $$T_{a}\otimes T_{p-1}^{(r-1)}\;\simeq\;T_{a+ p^{r}-p^{r-1}}. $$
    Since $a+ p^{r}-p^{r-1}$ is at least $p^r-1$, this module is projective in $\Rep E$. We can thus conclude the proof by considering the tensor product with $T_a$ of the $(r-1)$-th Frobenius twist of the short exact sequence in \Cref{eqsesT}.
\end{proof}

\begin{example}\label{ex:Omega-Ti}
\begin{enumerate}
    \item For $p>2$, we have
    $$\Omega(T_{2p^{r-1}-2})\;\simeq\; T_{p^r-2}.$$
    \item If $p=2$, then
    $\Omega(T_i)\simeq T_i$ for all $i\ge 2^{r-1}-1$.
\end{enumerate}
    
\end{example}

\label{sec:supp}

 \begin{lemma}\label{LemDescSt1}
 Normalise the expression $\ulambda=[\lambda_1:\cdots:\lambda_r]\in\mP^{r-1}_f$ so that $\lambda_1=1$.
     Over $C_p^r <\bk^+$, we then have
     $$St_1=T_{p-1}\simeq k[X_1,X_2, \ldots, X_r]/(X_a-\sum_{i>0} \binom{\lambda_a}{i}X_1^i\mid 1<a\le r).$$
 \end{lemma}
 \begin{proof}
     We can realise $St_1=\nabla_{p-1}$ as the degree $p-1$ polynomials in $k[x,y]$, where the two generators act by
     $$g_1: f(x,y)\mapsto f(x+y,y)\,\mbox{ and }\, g_a: f(x,y)\mapsto f(x+\lambda_a y,y),$$
     for $a>1$.
     In particular, the action of $g_a$ on $St_1$ depends polynomially on $\lambda_a$, with maximal degree $p-1$. As we can determine polynomials of degree $\le p-1$ by their values on~$\mF_p$, more concretely, the action of $g_a$ equals the action of 
     $$\sum_{j=0}^{p-1}p_j(\lambda_a) g_1^j$$
     for polynomials $p_j$ of degree $p-1$ satisfying $p_j(i)=\delta_{i,j}$ for $0\le i,j<p$. This can be rewritten to show that the action of $X_2$ equals the action of
     $$\sum_{j\ge 0}\left(\sum_{i=j}^{p-1}\binom{i}{j}p_i(\lambda_a)\right)X_1^j-1\;=\;\sum_{j>0}\binom{\lambda_a}{j}X_1^j.$$
     Since $St_1$ is a cyclic module, see \Cref{LemCyc}, and has dimension $p$, the conclusion follows.
 \end{proof}
 \begin{remark}
The lemma can be extended to describe $T_i$ for $i<p$, but we do not require this.
 \end{remark}

\begin{proposition}\label{PropSupp}
    The support of $SL_2$-tilting modules considered over $C_p^r$ via $\ulambda\in\mP^{r-1}_f(\bk)$ is given as follows. For $j\in\mN$ and 
    $$p^j-1\;\le\; i\;<\; p^{j+1}-1$$
    we have
    $$\supp(T_i)\;=\;\{\underline{\alpha}\in\mP^{r-1}\mid \sum_{s=1}^r \alpha_s \lambda_s^{p^l}=0\text{ for all }0\le l< j\}.$$
\end{proposition}

Before giving the proof we clarify the statement with some examples. Note that the support in \Cref{PropSupp} can be interpreted via an intersection of hyperplanes determined by some of the rows of the Moore matrix $\mathbb{M}(\ulambda)$, as in \Cref{Lem:Moore}.

\begin{example}\label{ExSupp}
    \begin{enumerate}
        \item We have $\supp T_i=\mP^{r-1}(\bk)$ for $i<p-1$.
        \item The support of $T_i$ for $p-1\le i<p^2-1$ is the hyperplane in $\mP^{r-1}$ determined by $\ulambda$.
        \item The support of $T_i$ for $p^{r-1}-1\le i<p^r-1$ is a single point $\blambda$. For example, we have
        $$\blambda\;=\;\begin{cases}
        [\lambda_2:-\lambda_1],&\text{if } r=2,\\
        [\lambda_2\lambda_3^p-\lambda_3\lambda_2^p:\lambda_3\lambda_1^p-\lambda_1\lambda_3^p:\lambda_1\lambda_2^p-\lambda_2\lambda_1^p],&\text{if } r=3.
        \end{cases}$$
        \item We have $\supp T_i=\varnothing$ for $i\ge p^r-1$, as $\mathbb{M}(\ulambda)$ is invertible.
    \end{enumerate}
\end{example}

For future reference, we reformulate parts (3) and (4) of the example as the following corollary.

\begin{corollary}\label{Cor:Supp}
    Assume $\ulambda\in\mP^{r-1}_f(\bk)$. Then the support of $St_{r-1}$ is given by the point $\blambda\in\mP^{r-1}(\bk)$, uniquely determined by
    $$\sum_{i=1}^r\lambda_i^{p^l}\blambda_i=0,\qquad\mbox{for all}\quad 0\le l\le r-2.$$
    Moreover, $\sum_{i=1}^r\lambda_i^{p^{r-1}}\blambda_i\not=0$.
\end{corollary}

\begin{proof}[Proof of \Cref{PropSupp}]
    Using the ideal structure of $\Tilt SL_2$, see \cite{Selecta}*{\S 5.3}, it suffices to prove the claim for the specific tilting modules
    $$St_j\;=\; T_{p^j-1}\;\simeq\; \bigotimes_{l=0}^{j-1} T_{p-1}^{(l)},$$
    for $j\in\mZ_{>0}$. Note that we thus ignore the trivial case `$St_0=\unit$'. It follows from \Cref{LemDescSt1} and Carlson's lemma (see for instance \cite{Benson-vector-bundles}*{Lemma~1.9.8}) that
    $$\supp T_{p-1}\;=\;\{\underline{\alpha}\mid \sum_{s=1}^m \alpha_s \lambda_s=0\}.$$
    From this we can easily derive 
    $$\supp T_{p-1}^{(l)}\;=\;\{\underline{\alpha}\mid \sum_{s=1}^m \alpha_s \lambda^{p^l}_s=0\}.$$
    The conclusion thus follows from the tensor product property of support in \eqref{sup:tensor}.
\end{proof}

\begin{remark} \label{rem:SX-projective} It immediately follow that $St_{r-1}\o X$ is a projective module for any $E$-module $X$ whose support does not contain $\blambda$ as in \Cref{Cor:Supp}.
\end{remark}

\begin{lemma}\label{lem:uniqueselfdual}
    If $r=2$, then $St_1$ is the unique $p$-dimensional self-dual periodic module with support~$\{\blambda\}$.
\end{lemma}
\begin{proof}
We already know that $St_1$ satisfies these properties. Consider an arbitrary periodic $p$-dimensional $C_p^2$-representation $M$. Since $X_1$ acts freely on $M$, we can identify $M$ with $\bk[X_1]/X_1^p$. The isomorphism class of $M$ is then completely determined by the action of $X_2$ on $1$. In other words, the action of $X_2$ of $M$ is equal to the action of
$$\sum_{i=1}^{p-1}a_iX_1^i$$
and the numbers $\{a_1,\ldots, a_{p-1}\}$ determine the isomorphism class of $M$. Furthermore, $a_1$ determines the support of $M$. Now, even without carrying out the computation, one can observe that demanding that the corresponding relation between $X_2$ and $X_1$ on $M^\ast$ will be the same results in a system of $p-1$ polynomial equations in the coefficients $a_i$ so that a unique solution will exist that writes $a_i$, for $i>1$, as a given polynomial in $a_1,\ldots, a_{i-1}$. Hence, for each support, there is only one self-dual module, which thus must be $St_1$.
\end{proof}

\subsection{Self-extensions of the Steinberg modules}

\begin{lemma}\label{lem:self-ext1}
    Consider $i,j\in\mN$ with $i\le p^{j+1}-1$ and let $H$ be an abstract subgroup of $U(\bk)$. Then the following conditions are equivalent:
\begin{enumerate}
    \item $T_{p^j+i-1}$ is a self-extension of $St_{j}=T_{p^j-1}$ over $U$;
    \item $T_{p^j+i-1}$ is a self-extension of $St_{j}=T_{p^j-1}$ over $H$;
    \item $\dim_{\bk}T_{p^j+i-1}=2\dim_{\bk} St_j$. 
\end{enumerate}

\end{lemma}
\begin{proof}
Clearly, (1) implies (2) and (2) implies (3), so we prove that (3) implies (1).

Let $L_i$ be the simple $SL_2$-representation of highest weight $i$. Firstly, by \cite{BEO}*{Lemma~3.3} or \cite{AbEnv}*{Lemma~4.3.4}, $St_j\otimes L_i$, as an $SL_2$-representation, is tilting.
Hence $T_{p^j+i-1}$ is an $SL_2$-summand of $St_{j}\otimes L_i$.

Denote by $e_+$ a vector in the simple module $L_i$ of weight $i$ and by $e_-$ a vector of weight $-i$. Then $e_-$ spans a $U$-submodule and $e_+$ is a generator of $L_i$ over $U$. Thus we have a $U$-linear inclusion and surjection
$$St_{j}\simeq St_{j}\otimes e_-\;\hookrightarrow\; St_{j}\otimes L_i\;\tto\; St_{j}\otimes e_+\simeq St_{j},$$
which compose to zero, induced from the socle and top $\unit\hookrightarrow L_i\tto \unit$. If we assume condition (3), then to prove condition (a) it suffices to show that this inclusion and surjection remain an inclusion and surjection when we restrict the middle term to the summand $T_{p^j+i-1}$. Moreover, by duality it suffices to prove one of these claims.

To conclude the proof we thus show that the composite
$$T_{p^j+i-1}\subset St_{j}\otimes L_i\tto St_j$$
is surjective. Since $St_j=L_{p^j-1}$ is generated, as an $U$-representation, by its highest weight vector $v_+$, it suffices to show that this highest weight vector is in the image. This follows immediately, because $v_+\otimes e_+$, which is of weight $p^j+i-1$, belongs to $T_{p^j+i-1}$.
\end{proof}

\begin{corollary}
    \label{cor:self-ext1}
    The following conditions are equivalent on $l\in\mN$, for $E=C_p^r$.
    \begin{enumerate}
        \item $T_l$ is a self-extension of $S=St_{r-1}$ over $E$;
        \item $l=p^{r-1}+ap^i-1$ for $0\le i<r-1$ and $0<a<p$, or $l=2p^{r-1}-1$.
    \end{enumerate}
    In particular, there are $pr-p-r+2$ isomorphism classes of indecomposable modules in $\cT$ that are self-extensions of $S$.
\end{corollary}
\begin{proof}
    This is an application of \Cref{lem:self-ext1} and Donkin's formula.

Assume first that $p>2$. It follows from the dimension formula in \cite{STWZ}*{Section~2} that the only $SL_2$-tilting modules of dimension $2p^{r-1}=2\dim St_{r-1}$ are $T_l$ for $l$ as in (2). The conclusion thus follows from \Cref{lem:self-ext1}.

For $p=2$, the same dimension formula shows that the only $SL_2$-tilting modules of dimension $2^r=2\dim St_{r-1}$ are 
$$T_{2^a+2^{i_1}+\cdots+2^{i_b}-1}$$
with $a+b=r$ and $a>i_1>i_2>\cdots$. By the results in \Cref{sec:supp}, a necessary condition for such modules to be self-extensions of $St_{r-1}$ over $E$ is that they have support contained in a singleton which in turn requires $a\ge r-1$ by \Cref{PropSupp}. These cases yield precisely $T_l$ with $l$ as in (2), so the conclusion follows as for $p>2$.
\end{proof}

\begin{remark} \label{rem::extensions}
    We can derive more extensions from the above results. For example
    \begin{enumerate}
        \item The modules $T_{mp},\dots,T_{mp+p-2}$ are self-extensions of $T_{mp-1}$ over $E=C_p^2$, for any $1\le m<p$. Indeed, by Donkin's formula, we can reduce this to the case $m=1$, which is contained in \Cref{cor:self-ext1}.
        \item That $T_l$, with $l$ as in \Cref{cor:self-ext1}(2), is a self-extension of $St_{r-1}$ remains true over any~$C_p^{r'}$.
    \end{enumerate}
\end{remark}

\begin{remark}\label{rem-ex}
\begin{enumerate}
    \item The self-extensions of $S$ in \Cref{cor:self-ext1} for the cases `$a=1$' are of the form
    $$T_{p^{r-1}+p^i-1}\;\simeq\; S\otimes V^{(i)}.$$
    For $p=2$, these are all the self-extensions in $\cT$, but not for $p>2$. In other words, for $p>2$, there are self-extensions of $S$ in $\cT$ which do not come from
    $$\Ext^1_E(\unit,\unit)\;\xrightarrow{S\otimes -}\;\Ext^1_E(S,S),$$
    see \Cref{ex-moreex}.
    \item For $r=2$, all self-extensions of $S$ in $\Rep E$ are (restrictions of) tilting modules. For $r>2$ this is no longer the case.
\end{enumerate}
   
\end{remark}

\subsection{Tensor powers of the Steinberg modules}
Recall the tensor ideal $\cT_r$ in the category $\cT\subset \Rep E$ of restrictions of $SL_2$-tilting modules from \Cref{SecFF}.
\begin{lemma}\label{Lem:TOT}
The ideal $\cT_r\subset\cT$ comprises the direct sums of direct summands of non-empty tensor products of $S=St_{r-1}$ and $\Omega S$. More precisely:
\begin{enumerate}
    \item If $p=2$, then the isomorphism classes of indecomposable $E$-modules appearing as summands in $S^{\otimes 3}$, resp. $S^{\otimes 4}$, are
    $$\{T_{2i-1}\mid 2^{r-2}\le i\le 2^{r-1}\}\quad\mbox{resp.}\quad \{T_{2i}\mid 2^{r-2}\le i< 2^{r-1}\}.$$
    \item If $p>2$, then the isomorphism classes of indecomposable $E$-modules appearing as summands in $S^{\otimes p+1}$, resp. $\Omega(S^{\otimes p+1})$, are
    $$\{T_{2i}\mid \frac{p^{r-1}-1}{2}\le i\le \frac{p^{r}-1}{2}\}\quad\mbox{resp.}\quad \{T_{2i+1}\mid \frac{p^{r-1}-1}{2}\le i< \frac{p^{r}-1}{2}\}.$$
\end{enumerate}

\end{lemma}
\begin{proof}
    All computations below are over $SL_2$.

Consider first $p=2$. In this case $S=T_{2^{r-1}-1}$ and $S^{\otimes 2}=T_{2^{r}-2}$. It then follows from Donkin's tensor product theorem that, up to multiplicities,
    $$S^{\otimes 3}\;\simeq\; \bigoplus_{0\le i_j\le 1} \left(T_{1+2i_0}\otimes T_{1+2i_1}^{(1)}\otimes \cdots\otimes T_{1+2i_{r-1}}^{(r-1)}\right).$$
    By considering the highest weight of each summand if follows that each $T_j$, with $j$ odd and $2^{r-1}-1\le j\le 2^r-2$, is a summand of $S^{\otimes 3}$.

    The above summands of $S^{\otimes 3}$ can, again by Donkin's formula, be written as
    $$T_1\otimes T_{a_1}^{(1)}\otimes \cdots \otimes T_{a_{r-1}}^{(r-1)},$$
where each $a_j$ is either $1$ or $2$. Taking the tensor product of this tilting module with $S$ yields a summand
    $$T_2\otimes T_{a_1'}^{(1)}\otimes \cdots\otimes  T_{a_{r-1}'}^{(r-1)},$$
    where $a'_j=3-a_j$ (in other words, $a'_j=2$ when $a_j=1$ and vice versa).
The above tilting modules are precisely $T_j$ for $j$ even and $2^{r-1}-1\le j\le 2^r-2$, and these thus appear as summands in $S^{\otimes 4}$.

That no other summands appear follows from parity, highest weight and the tensor ideal structure.

Now consider $p>2$. A direct computation shows that
$$T_{p-1}^{\otimes 2}\;\simeq\; \bigoplus_{a=0}^{(p-1)/2}T_{p-1+2a}.$$
It then follows from direct highest weight consideration that the summands in $T_{p-1}^{\otimes i+1}$ are precisely $T_{p-1+2a}$ for $0\le a \le i(p-1)/2$.
It thus follows that $S^{\otimes (p+1)}$ has summands (which are thus tilting modules)
$$T_{p-1+2a_0}\otimes T_{p-1+2a_1}^{(1)}\otimes \cdots\otimes T_{p-1+2a_{r-1}}^{(r-1)}$$
for all $0\le a_j\le p(p-1)/2$. The highest weights of these tilting modules range over
$$p^{r-1}-1+2a,\quad 0\le a\le p\frac{p^{r-1}-1}{2}.$$
The result for $\Omega(S^{\otimes p+1})$ then follows from \Cref{prop:Omega-Ti}.
\end{proof}

\subsection{The annihilator of the Steinberg module}
We investigate the action of elements of the annihilator ideal $\Ann_{\bk E}S$ of the Steinberg module $S=St_{r-1}$ in the group algebra $\bk E$.

\begin{lemma}\label{Lem:UseSL2}
    For $i,j\in\mN$ with $j\ge i$, we have an epimorphism $\Delta_j\tto \Delta_j$ of $U$-representations, so in particular of $E$-representations. The kernel of the morphism is spanned by $j-i$ weight vectors of lowest weights.
\end{lemma}
\begin{proof}
    We can easily describe the dual monomorphism
    $$\bk[x,y]_i=\nabla_i\hookrightarrow \nabla_j=\bk[x,y]_j,\quad x^{i-a}y^a\mapsto x^{i-a}y^{a+j-i},$$
    from which the result follows.
\end{proof}

For the following lemma, recall our conventions from \ref{basis}.
\begin{lemma}\label{Lemxi0}
    Fix $\xi\in\Ann_{\bk E}S$. Consider $t\in\mZ_{>0}$ and $a_i\in\mN$ for $1\le i\le t$. The element $\xi$ sends $e_1\otimes\cdots \otimes e_1$ in $\otimes_{i=1}^tSt_1^{(a_i)}$ to a linear combination of
        $$e_{b_1}\otimes e_{b_2}\otimes \cdots\otimes e_{b_t}\quad\mbox{with}\quad \sum_{i=1}^t (b_i-1)p^{a_i}\ge p^{r-1}.$$

\end{lemma}
\begin{proof}
 We set $a:=(p-1)\sum_{i=1}^tp^{a_i}$, the maximal possible value of $\sum_{i=1}^t (b_i-1)p^{a_i}$. Using the highest weight structure of $\Rep SL_2$, we know that there is an $SL_2$-linear morphism 
    $$\Delta_a\;\to\; \bigotimes_{i=1}^tSt_1^{(a_i)}$$
    that sends the highest weight vector of $\Delta_a$ to $e_1\otimes \cdots \otimes e_1$.

    If $a\le p^{r-1}-1$, it follows that $\xi\in \Ann S$ acts trivially on the quotient $\Delta_a$ of $S=\Delta_{p^{r-1}-1}$, see \Cref{Lem:UseSL2}. Hence $\xi$ acts trivially  on $e_1\otimes\cdots \otimes e_1$ as predicted.

    Conversely, assume $a\ge p^{r-1}$. It now follows from \Cref{Lem:UseSL2} that $\xi$ sends the highest weight vector of $\Delta_a$ to vectors of weights strictly below $a-2(p^{r-1}-1)$, from which the claim follows.
\end{proof}

\subsubsection{}\label{defxi}Since $S\otimes St_1^{(r-1)}\simeq \bk E$, there is a unique $\xi\in\bk E$ for which 
$$\xi(e_1\otimes \cdots \otimes e_1)\;=\;e_1\otimes \cdots\otimes e_1\otimes e_2\;\in\; S\otimes St_1^{(r-1)}.$$
Since $S$ is cyclic, $\xi$ is in $\Ann_{\bk E}S$.

\begin{lemma}\label{Powxi}
    Let $\xi\in\bk E$ be as in \ref{defxi}.
    \begin{enumerate}
        \item For every $1\le i\le p$, we have
        $$\xi^{i-1}(e_1\otimes \cdots \otimes e_1)\;=\; e_1\otimes \cdots\otimes e_1\otimes e_i\;\in\; S\otimes St_1^{(r-1)}.$$
        \item For every $0\le i<r-1$ and $a,b\in\mN$, we have
        $$\xi(\rad^a S\otimes \rad^b St_1^{(i)})\;\subset\; \rad^{a+1} S\otimes \rad^{b+1} St_1^{(i)}$$
    \end{enumerate}
\end{lemma}
\begin{proof}
    Part (1) can be proved by induction on $i$, since we can interpret, via \ref{EndS1}, the submodule of $St_1^{(r-1)}$ spanned by $\{e_2,\ldots, e_p\}$ as the quotient $St_1^{(r-1)}/\bk e_p$, where $e_{l+1}$ in the former corresponds to the equivalence class of $e_l$.

    It is a general property that for two $E$-representations $U_1,U_2$ and $\eta\in \Ann(U_1)\cap \Ann(U_2)$, 
    $$\eta (\rad^a U_1\otimes \rad^b U_2)\;\subset\;\rad^{a+1} U_1\otimes \rad^{b+1} U_2.$$
Indeed, without loss of generality we can assume that $U_1\not=0$, so that $\eta\in \rad \bk E$. It then follows that
$$\Delta(\eta)\;=\; \eta\otimes 1+1\otimes \eta+ \sum_i\eta^1_i\otimes \eta^2_i$$
    with $\eta^j_i\in \rad \bk E$. Applying this expression to $U_1\otimes U_2$ demonstrates the claim.

    Since $\xi$ acts trivially on $S$, it also acts trivially on its quotient $St_1^{(i)}$, for $i\le r-2$. 
\end{proof}

\section{Extensions between specific uniserial modules} \label{sec:ext}

The main goal of \Cref{sec::subcategories} below is to compute the image of the map
\begin{equation} \label{eq:psi}
\omega_{U,W} = S\o-: \Ext^1_E( U,W)\to\Ext^1_E(S\o U,S\o W)    
\end{equation}
for $S=St_{r-1}$ the highest Steinberg module in $\Rep E=\Rep C_p^r$, determined by $\ulambda\in\mA^r_f$ as before, and for modules $U,W$ in a full abelian subcategory $\cC\subset\Rep E$ that is equivalent to $\Rep C_p$ such that $S\o Q\simeq\kk E$, for $Q$ the indecomposable projective in $\cC$. 

The description of the image of $\psi$ will be used in subsequent sections. As an immediate consequence, we will also obtain the following result:

\begin{theorem} \label{thm:S-tensor-ext-uniserials} Let $U$ be a uniserial $p$-dimensional representation with $\blambda\not\in\supp(U)$. Then for any extension $X$ of two quotients of $U$, we have $S\otimes X\in\cT$.
\end{theorem}

\begin{proof} See \Cref{cor:S-tensor-ext-Ud-Ue}.
\end{proof}

\subsection{Parametrising extensions} In this subsection, we describe the extensions in the source of $\psi$, see \Cref{eq:psi}.

To alleviate notation, we fix a generator $1_X=1$ for any cyclic module $X$. For any uniserial module $U$, such as $St_1$, we additionally fix an endomorphism $z_U=z$ with a one-dimensional kernel. By abuse of notation, this allows us to denote a basis of $U$ by 
 $$(z_U^i(1_U))_{0\le i< \dim U}=(z_U^i)_{0\le i< \dim U}=(z^i)_{0\le i< \dim U}.$$ 

\subsubsection{} Fix some $1\le k\le r$. Assume $\cC\subset\Rep E$ is a topologising subcategory equivalent to $\Rep C_p^{r-k}$. Let $Q$ be the indecomposable projective object of $\cC$. Assume also $\cC'\subset\Rep E$ is another topologising subcategory such that $\cC'\simeq\Rep C_p^k$, and let $Q'$ be its indecomposable projective object. Assume $Q\o Q'\simeq\kk E$ in $\Rep E$. 

We will write $\Omega_{\cC}$ for the shift inside $\cC$.

\begin{remark} The assumptions on $\cC$ imply that it is the full subcategory on those objects of $\Rep E$ which are subobjects of $Q^m$ for some $m\ge1$, or equivalently, quotient objects of $Q^m$ for some $m\ge1$.
Hence, $\one\in\Rep E$ is the unique simple object of $\cC$ up to isomorphism.
\end{remark}

\begin{example} (a) Let $\cC$ be the full subcategory on those modules on which $X_1,\dots,X_k\in\kk E$ all act as zero. Then $\cC$ satisfies the above assumptions. Furthermore, we could let $\cC'$ be the full subcategory on those modules on which $X_{k+1},\dots,X_r\in\kk E$ all act as zero

(b) Let $\cC$ be the full abelian (or equivalently, topologising) subcategory generated by $St_{r-k}$. Then $\cC$ satisfies the above assumptions. For $\cC'$ we can take the full abelian subcategory generated by $St_{1}^{(r-k)}\otimes \cdots St_1^{(r)}$.
\end{example}

\begin{definition} For any $m\ge1$ and any epimorphism $\pi:Q^m\to U$ in $\cC$, we define two subspaces of $Q^m\o Q'$,
$$
M_\pi:=\ker\pi \o Q'+Q^m\o \rad(Q') , 
$$
$$
\text{and}\quad
K_\pi:=\ker\pi\o\rad(Q')+Q^m\o\rad^2(Q') .
$$
\end{definition}
The following results can be viewed as a concrete realisation of the Künneth formula
$$
\Ext^1_E(U,W)\;\simeq\;\left(\Ext^1_{\cC}(U,W)\otimes \Hom_{\cC'}(\unit,\unit)\right)\oplus \left(\Hom_{\cC}(U,W)\otimes \Ext^1_{\cC'}(\unit,\unit)\right).
$$

\begin{lemma} \label{lem::extensions-from-subcategory} For any $U\in\cC$, there is a morphism $\pi_U\in\stHom(\Omega U,\Omega_{\cC} U\oplus U^k)$ such that the induced map
$$
\stHom(\Omega_{\cC}U\oplus U^k,W)\to \stHom(\Omega U,W),\quad
f\mapsto f \pi_U ,
$$
is an isomorphism for all $W\in\cC$.

More concretely, assume $\pi:Q^m\to U$ is an epimorphism for some $m\ge1$, and set $M:=M_\pi$ and $K:=K_\pi$. Then $M$ is stably isomorphic to $\Omega U$, $M/K\simeq\ker\pi\oplus U^k$, and $\pi_U$ can be taken to be the natural epimorphism $M\to M/K$.
\end{lemma}

\begin{proof}

By the assumptions, $\ker\pi$ is stably isomorphic to $\Omega_\cC U$ in the stable category of $\cC$. We have a short exact sequence
$$
0\to M\to Q^m\o Q'\xrightarrow{\pi\o (Q'\twoheadrightarrow\one)} U\to 0
$$
with $M=\ker\pi \o Q'+Q^m\o \rad(Q')$. Hence $M$ is stably isomorphic to $\Omega U$.

Under the equivalence $\cC'\simeq\Rep C_p^k$, $Q'$ can be viewed as $\kk[X_1,\dots,X_k]/(X_i^p)$. Let $z_i$ be the endomorphism of $Q'$ sending $1\mapsto X_i$. By assumption, $Q\o Q'\simeq\kk E$. Hence we can pick $(\eta_i)_i\in\kk E$ such that $\eta_i(1\o 1)=1\o X_i$ in $Q\o Q'$ for all $i$. Then $\eta_i|_{Q\o Q'}=Q\o z_i$, so 
$$
\sum \eta_i(M)
= \sum \eta_i(\ker\pi\o Q' + Q^m\o \rad(Q'))
= \ker\pi\o \rad(Q') + Q^m\o\rad^2(Q')
= K
.
$$

Also, $\eta_i Q\simeq \eta_i(Q\o \soc(Q'))\simeq 0$ for all $i$. This implies $\eta_i|_W=0$ for all $W\in\cC$ and all $i$. This means any morphism in $\Hom(M,W)$ factors uniquely through $M/K$. 

To describe $M/K$, note that the intersection of the subspaces $\ker\pi\o Q'$ and $Q^m\o\rad(Q')$ of $Q\o Q'$ is $\ker\pi\o\rad(Q')$. Hence, their images intersect trivially already in the quotient space $Q\o Q' / (\ker\pi\o\rad(Q')$, which implies that $M/K$ is isomorphic to
$$
\frac{\ker\pi\o Q'}{(\ker\pi\o\rad(Q')+Q^m\o\rad^2(Q'))\cap( \ker\pi\o Q')}
\oplus
\frac{Q^m\o\rad(Q')}{\ker\pi\o\rad(Q')+Q^m\o\rad^2(Q')} ,
$$
which in turn is isomorphic to
$$\ker\pi\o \frac{Q'}{\rad(Q')}
\oplus
\frac{Q^m}{\ker\pi}\o\frac{\rad(Q')}{\rad^2(Q')}
\cong \ker\pi \oplus U^k ,
$$
as asserted. 
\end{proof}

\begin{corollary} \label{cor::R-o-Hom} In the situation of \Cref{lem::extensions-from-subcategory}, let $R\in\Rep E$ be any module such that $R\o Q$ is projective in $\Rep E$. Consider $U,W\in\cC$ and the map
$$
\psi:R\o-:\stHom(\Omega U,W)\to\stHom(R\o\Omega U,R\o W). 
$$
Then with $\pi_U$ as in \Cref{lem::extensions-from-subcategory}, the induced map
$$
\stHom_\cC(\Omega_{\cC}U\oplus U^k,W)\to \im(\psi),\quad
f\mapsto R\o f \pi_U ,
$$
is an epimorphism.
\end{corollary}

We can now specialize $k$ to $r-1$ to obtain concrete descriptions of $\im(\psi)$. For all $1\le d\le p$, we pick a representative $U_d$ for the isomorphism class of indecomposable objects of dimension $d$ in $\cC\simeq\Rep C_p$. Note that $Q\simeq U_p$ and there are short exact sequences
$$
0\to U_{p-d}\to U_p\to U_d\to 0
$$
for all $d$, so $\Omega_\cC U_d\simeq U_{p-d}$.

For all $1\le d,e\le p$, we define an index set
$$
I_{d,e} := \{m\in\mZ\mid \max(0,e-d)\le m <\min(p-d,e)\} .
$$

\begin{lemma} \label{lem::Rep-Cp-stable-homs} For all $1\le d,e\le p$, the maps $(1\mapsto z^m)_{m\in I_{d,e}}$ correspond to a basis of $\stHom_\cC(U_d,U_e)$. 
\end{lemma}

\begin{proof} The maps $(1\mapsto z^m)$ as above for $\max(0,e-d)\le m\le e-1$ form a basis for $\Hom(U_d,U_e)$. Those with $m\ge p-d$ factor via the indecomposable projective $U_p$ in $\cC$. On the other hand, any projective object $Q$ in $\cC$ is a direct sum of copies of the uniserial module $U_p$, any map from $U_d$ to $Q$ has its image in $\soc_d(Q)=\rad_{p-d}(Q)$, so any map from $U_d$ to $U_e$ via $Q$ has its image in $\rad_{p-d}(U_e)$. This means the maps $(1\mapsto z^m)$ with $m\ge p-d$ span the space of morphism that factor via a projective object. This implies the assertion.
\end{proof}

\begin{definition} 
For any $1\le d\le p$ and $U:=U_d$, we fix a surjection $\pi:U_p\to U$ and an epimorphism $\pi_U: M_\pi\tto M_\pi/K_\pi\simeq \Omega_\cC U_d\oplus U_d^{r-1}\simeq U_{p-d}\oplus U_p^{r-1}$ as in \Cref{lem::extensions-from-subcategory}.
Let $\rho_0:U_{p-d}\oplus U_d^{r-1}\tto U_{p-d}$ and $(\rho_i:U_{p-d}\oplus U_d^{r-1}\tto U_d)_{1\le i\le r-1}$ be the structural projections of the direct sum.
We set
$$
f_{d,e,i,m} := (1\mapsto z^m) \rho_i \pi_U : M_\pi\to U_e .
$$
for all $1\le e\le p$, $0\le i\le r-1$, and $m\in\begin{cases} I_{p-d,e} & i=0 \\ I_{d,e} & 1\le i\le r-1 \end{cases}$.
\end{definition}

\begin{corollary} \label{lem::parametrization-stable-hom} If $k=r-1$, then in the situation of \Cref{lem::extensions-from-subcategory} and \Cref{cor::R-o-Hom} with $U=U_d$ and $W=U_e$ in $\cC\simeq\Rep C_p$, the space $\im(\psi)$ is spanned by the maps induced by the maps $(R\o f_{d,e,0,m})_{m\in I_{p-d,e}}$ and $(R\o f_{d,e,i,m})_{1\le i\le r-1,m\in I_{d,e}}$ just defined.
\end{corollary}

\begin{proof} This follows immediately from \Cref{cor::R-o-Hom} with \Cref{lem::Rep-Cp-stable-homs}.
\end{proof}

\subsection{Tensor products with Steinberg modules} \label{sec::subcategories}

\subsubsection{}\label{Qxi} Recall $S:=St_{r-1}=St\o St^{(1)}\o\dots St^{(r-2)}$. We fix a $Q\in\Rep E$ that generates a full topologising subcategory equivalent to $\Rep C_p$ and that satisfies $S\o Q\simeq\kk E$.

Recall that for any uniserial module $U$, such as $St$ and $Q$, we have fixed a generator $1_U=1$ and an endomorphism $z_U=z$ with a one-dimensional kernel, yielding a basis $(z_U^i)_{0\le i<\dim U}=(z^i)_{0\le i< \dim U}$. 

Let $z_i:=St\o\dots\o z\o St^{(i)}\o\dots\o St^{(r-2)}$ be the corresponding endomorphism of $S$ for all $1\le i\le r-1$. Set $\psi'_i:=z_{i+1}^{p-1}\dots z_{r-1}^{p-1}$ for all $0\le i\le r-1$, so $\psi'_{r-1}=\id_S$.

As $S\o Q\simeq\kk E$, we can pick $\xi\in\kk E$ such that $\xi(1\o 1)=1\o z$ in $S\o Q$. Then $\xi|_{S\o Q}=S\o z$ and $\xi|_S=0$. We fix such a $\xi$.

\begin{remark} It follows immediately that any $Q$ as in \ref{Qxi} is uniserial of dimension $p$, and its support does not contain the one point which is in the support of $S$ (see \Cref{sec:supp}).
\end{remark}

We will use the following idea for cyclic modules like $S$:

\begin{lemma}
Let $M$ be a cyclic $E$-module with an endomorphism $f$ and $i\in\mZ_{>0}$. Then $f\in\rad^i(\End(M))$ if and only if $\im(f)\subset\rad^i(M)$.
\end{lemma}

\begin{proof}
Assume $\im(f)$ lies in $\rad^i(M)$. Let $m$ be a generator of $M$. Then there is an element $y\in\rad^i(\kk E)$ such that $f(m)=ym$. This implies $f=y|_M$. Furthermore, $y$ is a sum of $i$-fold products of elements from $\rad(\kk E)$. Elements from $\rad(\kk E)$ are nilpotent, hence their action on $M$ lies in the radical of $\End(M)$. This means $f$ lies in $\rad^i(\End(M))$.

Now assume $f$ lies in $\rad^i(\End(M))$. Then $f$ is a sum of $i$-fold products of elements from $\rad(\End(M))$. As $M$ is cyclic, each element in $\rad(\End(M))$ is given by the action of an element in $\kk E$, which will have to lie in $\rad(\kk E)$, as otherwise, it would be invertible in $\kk E$. This means $\im(f)$ lies in $\rad^i(M)$.
\end{proof}

One of the most important properties of the Steinberg module $S$ for our purposes is the following result. We are interested in $S\otimes S$ as we will realise $S\otimes \Omega X$, for the relevant $X$, as a submodule in $S\otimes S\otimes X$. $\xi$ is as in \Cref{Qxi}.

\begin{lemma} \label{lem::xi-S-S}
There are automorphisms $(u_i)_{0\le i\le r-2}$ of $S$ such that
$$
\xi|_{S\o S} \in \sum_{0\le i\le r-2} u_i \psi'_i \o z_{i+1} + S\o\rad^2(\End(S)) .
$$    
\end{lemma}

\begin{proof}
As a special case \Cref{Lemxi0} we have
$$
\xi(1\o 1) \in \sum (z_1^{a_0}\dots z_{r-1}^{a_{r-2}}\o z_1^{b_0}\dots z_{r-1}^{b_{r-2}}) (1\o 1)
$$
for those tuples $(a_i)_i, (b_i)_i$ in $\{0,\dots,p-1\}$ that satisfy
$$
\sum_{0\le i\le r-2} (a_i+b_i) p^i  \ge p^{r-1} .
$$
The tuples corresponding to elements in $S\o\rad^2(S)$ are those with $\sum b_i>1$. The remaining tuples satisfying the condition correspond exactly to elements $(m_i(z_1,\dots,z_i) \psi'_i\o z_{i+1})(1\o1)$ for some $i$ and some monomials $m_i$. Hence,
$$
\xi(1\o1) \in \sum_{0\le i\le r-2} (u_i(z_1,\dots,z_i) \psi'_i \o z_{i+1})(1\o1) + S\o\rad^2(S)  
$$
for some polynomials $u_i$.

Assume, in order to find a contradiction, that $u_j$ has constant term $0$ for some $j$.  Consider the endomorphism 
$$
f := (z_1\dots z_j)^{p-1}\o (z_1\dots z_j z_{j+2}\dots z_{r-1})^{p-1} z_{j+1}^{p-2}
$$
of $S\o S$. As $\im(z_{St^{(i)}}^{p-1})\simeq\one$ and $\im(z_{St^{(i)}}^{p-2})\simeq V^{(i)}$, the image of $f$ is isomorphic to $St^{(j)}\o\dots\o St^{(r-2)}\o V^{(j)}$, but the above formula for $\xi$ and the assumption that $u_j$ has constant term zero imply $\xi|_{S\o S}\, f=0$. So $\xi$ acts as zero on $St^{(j)}\o\dots\o St^{(r-2)}\o V^{(j)}$. But then $\xi$ acts as zero on $V^{(r-1)}$, which is a subquotient of the latter module, as can be seen via the inclusions and surjections of $SL_2$-modules 
$$
St_1^{(j)}\otimes St_1^{(j+1)}\otimes \cdots \otimes St^{(r-2)}_1\otimes V^{(j)}\simeq T_{p^{r-1-j}}^{(j)}
\;\hookleftarrow\; \Delta_{p^{r-1-j}}^{(j)} \;\tto\;  L_{p^{r-1-j}}^{(j)}\simeq V^{(r-1)} .
$$
$\xi$ acting as zero on $V^{(r-1)}$ implies that $\xi$ sends $St^{(r-1)}$ to its second radical. But then it would follow that $\xi$ sends 
$$\kk E\simeq S\o St^{(r-1)}\simeq S\otimes Q$$ to its second radical, which is a contradiction, as $\xi_{S\o Q}(1\o1)=1\o z$ by assumption.

Hence, the constant terms of all $u_i$ are non-zero, so they are automorphisms of $S$.
\end{proof}

\begin{definition}
Set $$\psi_{r-1}:=\psi'_{r-1}=\id_S,\quad \psi_i:=u_i \psi'_i$$ with $u_i$ as in \Cref{lem::xi-S-S} for all $0\le i\le r-2$.
\end{definition}

\begin{corollary} \label{lem::image-xi-d} For all $1\le d\le p$,
$$
\xi^d|_{S\o Q\o S} \in \sum_{0\le i\le r-2} (\psi_i\o \rad^{d-1}(\End(Q))\o z_{i+1}) + S\o \rad^d(\End(Q))\o S + S\o Q\o \rad^2(\End(S)) .
$$
\end{corollary}

\begin{proof} \Cref{lem::xi-S-S} proves the case $d=1$. The cases $d>1$ follow by induction in $d$ using that $\psi_i\psi_j=0$ for all $0\le i,j\le r-2$.
\end{proof}

Now let $\cC$ and $\cC'$ be the full abelian subcategories of $\Rep E$ generated by $Q$ and $Q':=S$, respectively. Then $\cC\simeq\Rep C_p$ and $\cC'\simeq\Rep C_p^{r-1}$.

\begin{lemma} \label{lem::s-o-om-c-u} For any indecomposable $U\in\cC$,  $S\o\Omega_\cC U$ is stably isomorphic to $\Omega(S\o U)$.
Moreover, the stable isomorphism can be realised as follows: pick an epimorphism $\pi:Q\to U$ and set $d:=\dim(U)$. Then with $M=M_\pi\subset Q\o S$ as in \Cref{lem::extensions-from-subcategory}, the vector $\xi^d(1\o1\o1)$ in $S\o Q\o S$ lies in the submodule $S\o M$ and generates a direct summand isomorphic to $S\o\ker\pi$ with a projective complement in $S\o M$.
\end{lemma}

\begin{proof}
We have a commutative diagram
$$
\begin{tikzcd}[column sep=2cm]
S\o\Omega U
    \ar[r]
    \ar[d]
&
S\o Q\o S
    \ar[r,"S\o(1\o1 \mapsto 1)"]
    \ar[d,"S\o(1\o1 \mapsto 1)"]
&
S\o U
    \ar[d,"="]
\\
S\o\ker\pi
    \ar[r]
&
S\o Q
    \ar[r,"S\o\pi"]
&
S\o U
\end{tikzcd}
$$
where the rows are exact, and the left vertical arrow is determined by the remaining diagram. The middle vertical arrow splits, as $S\o Q$ is projective, and we may assume the splitting sends $1\o1\mapsto 1\o1\o1$. This induces a splitting of the left vertical arrow. As $S\otimes \ker\pi$ is generated by $\xi^d(1\o1)$ inside $S\o Q$, the image of $S\o\ker\pi$ under the splitting of the left vertical arrow is generated by $\xi^d(1\o1\o1)$.
\end{proof}

Recall that we have picked representatives $(U_d)_d$ for the isomorphism classes of indecomposable objects in $\cC$, labeled by their dimension $d$. Note that if $\pi:Q\to U_d$ is an epimorphism, then $\ker\pi\simeq U_{p-d}$. For all $1\le d,e\le p$, we consider the map 
$$
\omega_{d,e}:=S\o- : \stHom(\Omega U_d,U_e)\to\stHom(\Omega(S\o U_d),S\o U_e)\simeq\stHom(S\o U_{p-d},S\o U_e).
$$

Recall that, for all $1\le d,e\le p$,
$$
I_{d,e} = \{m\in\mZ\mid \max(0,e-d)\le m\le \min(p-d,e)-1\} .
$$

\begin{proposition} \label{prop::im-omega} $\im(\omega_{d,e})$ is spanned by the set $F_1\cup F_2$, where 
$$
F_1 := \{S\o (1\mapsto z^m)\}_{m\in I_{p-d,e}} 
\quad\text{and}\quad
F_2 := \begin{cases}
\{ \psi_i\o z^{d-1} \}_{0\le i\le r-2} & d=e \\
\emptyset & d\neq e
\end{cases}.
$$
Any non-zero element in $\im(\omega_{d,e})$ equals one of the elements in $F_1\cup F_2$ up to an automorphism of $S\o U_{p-d}$ and of $S\o U_e$.
\end{proposition}

\begin{proof} Set $U=U_d$ and $W=U_e$. Then we are in the set-up of \Cref{lem::parametrization-stable-hom} with $\im(\omega_{d,e})=\im(\psi)$. Using \Cref{lem::parametrization-stable-hom} and \Cref{lem::s-o-om-c-u}, all we have to show is that $(S\o f_{d,e,\ell,m})|_{\langle \xi^d(1\o1\o1) \rangle}$ is stably zero or one of the maps in $F_2$ for all $1\le \ell\le r-1$ and $m\in I_{d,e}$, depending on $d$ and $e$. Using \Cref{lem::image-xi-d}, we know that
$$
\xi^d(1\o1\o1) \in \sum_i (1\o\eta^{d-1})(\psi_i\o Q\o z_{i+1})(1\o1\o1) + S\o\rad^d(Q)\o S + S\o Q\o\rad^2(S).
$$
for $\eta\in\kk E$ such that $\eta(1\o1)=z\o1$ in $Q\o S$. Using the concrete formula for $f_{d,e,\ell,m}$ shows then that
$$
(S\o f_{d,e,\ell,m}) (\xi^d(1\o1\o1)) \in (1\o\eta^{d-1}) (\psi_{\ell-1}\o z^m)(1\o1)
= (\psi_{\ell-1}\o z^{m+d-1})(1\o1)
\quad\in S\o U_e
$$
which is $0$ as soon as $m>e-d$. This is true, in particular, if $d>e$. If $d\le e$, then the only non-zero cases correspond to $m=e-d$ and satisfy
$$
(S\o f_{d,e,\ell,e-d})( \xi^d(1\o1\o1) ) \in (\psi_{\ell-1}\o z^{e-1})(1\o1) \kk ,
$$
which yields the desired result if $d=e$. Finally, if $d<e$, then the map $U_{p-d}\to U_e, 1\mapsto z^{e-1},$ factors through the indecomposable projective $U_p$ in $\cC$. Hence, the map $S\o f_{d,e,\ell,e-d}$ viewed as a map from $\Omega(S\o U_p)\simeq S\o U_{p-d}$ to $S\o U_e$ is stably zero in this case, too. This shows the assertions.
\end{proof}

We can now identify the extensions that form the spanning set in \Cref{prop::im-omega}.

\begin{lemma} \label{lem::S-Ud}
For any $1\le d\le p$, we have isomorphisms
$$ S\o U_d \simeq S\o T_{d-1}^{(r-1)}\simeq T_{dp^{r-1}-1} .
$$    
\end{lemma}

\begin{proof} The second isomorphism is an application of Steinberg's tensor product theorem / Donkin's formula.

For the first isomorphisms, consider the map
$$
f: \kk E\simeq S\o U_p \to S\o T_{d-1}^{(r-1)}, \quad 1\o 1\mapsto 1\o 1.
$$
It is surjective, as the module on the right-hand side is a quotient of $S\o T_{p-1}^{(r-1)}\simeq S\o St^{(r-1)}\simeq\kk E$, so it is cyclic. We claim that $f \xi^d|_{S\o U_p}=0$. Indeed, the claim follows from noting that 
$$
\xi|_{S\o T_{d-1}^{(r-1)}} \in S\o\rad\End(T_{d-1}^{(r-1)}),
$$
as $\xi|_S=0$ and $\xi\in\rad(\kk E)$. From the claim it follows that $f$ factors through $S\o (U_p/z^d U_p)\simeq S\o U_d$. A dimension count then confirms that we have constructed an isomorphism as desired.
\end{proof}

\begin{lemma} \label{lem::identify-extensions}
In the setting of \Cref{prop::im-omega}, the maps $(S\o(1\mapsto z^m))_{m\in I_{p-d,e}}$ in $F_1$ correspond to the modules $S\o(U_m\oplus U_{d+e-m})\simeq T_{mp^{r-1}-1}\oplus T_{(d+e-m)p^{r-1}-1}$, viewed as extensions via the short exact sequences
$$
S\o \Big(U_e\xrightarrow{1\mapsto(-1,z^{d-m})} U_m\oplus U_{d+e-m}
\xrightarrow{(1,0)\mapsto z^{d-m}, (0,1)\mapsto 1} U_d \Big)
$$
(with the conventions , $T_{-1}:=0$, $U_0:=0$, $1_{U_0}:=0$). Assuming $d=e$, the maps $(\psi_i\o z^{d-1})_{0\le i\le r-2}$ in $F_2$ correspond to the modules $S\o U_d\o V^{(i)}\simeq T_{dp^{r-1}+p^i-1}$, viewed as extensions via the short exact sequences
$$
S\o U_d\o\Big( 
\one \xrightarrow{1\mapsto z} V^{(i)}\xrightarrow{1\mapsto 1} \one \Big).
$$
\end{lemma}

\begin{proof} Using that by \Cref{lem::S-Ud}, $S\o U_d\simeq T_{dp^{r-1}-1}$ and, consequently, $S\o U_d\o V^{(i)}\simeq T_{dp^{r-1}+p^i-1}$, the given extensions are non-isomorphic as modules. As the orbits of the maps in $F_1\cup F_2$ under the action of the automorphism groups contain $\im\omega_{d,e}$, and as there are as many extensions as elements in $F_1\cup F_2$, the extensions and maps are in bijection.

To identify which extension corresponds to each map, we first observe that for all $m\in I_{p-d,e}$, the diagram
$$
\begin{tikzcd}[column sep=3cm]
U_{p-d}
    \ar[r,"{1\mapsto z^d}"] 
    \ar[d,"{1\mapsto z^m}"]
& U_p
    \ar[r,"{1\mapsto 1}"]
    \ar[d,"{1\mapsto (-z^{m-d},1)}"]
& U_d
    \ar[d,"="]
\\
U_e
    \ar[r,"{1\mapsto(-1,z^{d-m})}"]
& U_{m} \oplus U_{d+e-m}
    \ar[r,"{(1,0)\mapsto z^{d-m}, (0,1)\mapsto 1}"]
& U_d
\end{tikzcd}
$$
commutes, has exact rows, and the left square is a pushout. Tensoring the diagram with $S$ shows the desired correspondence for the maps in $F_1$. This also implies that the maps in $F_2$ are in bijection with the remaining extensions.

To relate the remaining extensions with the maps in $F_2$, we realise $V^{(i)}$ as the quotient of $St_1^{(i)}$ via the map $1\mapsto 1$, see \Cref{Lem:UseSL2}. 
We specialise $Q=St_1^{(r-1)}$, which is possible as $S\o St_1^{(r-1)}\simeq\kk E$. Note that the relation between a map in $F_2$ and the self-extension of $S$ it determines is independent of the choice of $Q$. 
It follows from \Cref{Powxi} that the morphism 
$$\bk E\simeq S\otimes St_1^{(r-1)}\;\to\; S\otimes V^{(i)},\quad 1\o 1\mapsto 1\o 1 ,$$
factors through the quotient $S\otimes V^{(r-1)}$. Indeed, using \Cref{Powxi}(1) it suffices to verify that 
$$
1\o z^2 =  \xi^2(1\o1) \in S\otimes St_1^{(r-1)}$$
is sent to zero, which is implied by \Cref{Powxi}(2).
It follows that there exists a commutative diagram (whose rows are short exact sequences) 
$$
\begin{tikzcd}[column sep=3cm]
S 
    \ar[r,"1\mapsto1\o z"] 
    \ar[d,"\psi",dashed] 
& S\o V^{(r-1)} 
    \ar[r,"s\o z^k\mapsto s \delta_{k,0}"] 
    \ar[d,"1\o1\mapsto1\o1"] 
& S \ar[d,"="] 
\\
S 
    \ar[r,"1\mapsto1\o z"] 
& S\o V^{(i)} 
    \ar[r,"s\o z^k\mapsto s \delta_{k,0}"]
& S
\end{tikzcd}\qquad,
$$
where $\psi$ is uniquely defined by the identity $\psi(1)\o z=\xi(1\o 1)$ in $S\o V^{(i)}$. Now \Cref{lem::xi-S-S} shows that $\psi=\psi_i$, which identifies the remaining extensions as asserted.
\end{proof}

Combining \Cref{prop::im-omega} and \Cref{lem::identify-extensions}, we obtain the following implications for \Cref{conj-a}:

\begin{corollary} \label{cor:S-tensor-ext-Ud-Ue} For all $1\le d,e\le p$ and any extension $X$ in $\Rep E$ of $U_d$ and $U_e$, $X$ lies in $\widetilde\cT$, i.e., $S\o X$ lies in $\cT$.
\end{corollary}

This proves \Cref{thm:S-tensor-ext-uniserials}, as for any uniserial $p$-dimensional representation $U$ with $\blambda\not\in\supp(U)$, $U\o S\cong\kk E$.

Specialising $d=e=1$ in \Cref{prop::im-omega}, we obtain the following result which will be relevant in \Cref{sec:Carlson} and \Cref{sec:Loewy}:

\begin{corollary} \label{cor:im-omega-1-1}
(a) $\im(\omega_{1,1})\subset\stHom(\Omega S,S)\simeq\stHom(S\o U_{p-1},S)$ is spanned by $(\psi_i\o f_0)_{0\le i\le r-1}$,  where $f_0:U_{p-1}\to\one$ is any epimorphism.

(b) Every non-zero element in $\im(\omega_{1,1})$ is (uniquely) of the form $u\psi_i \o f_0$ for some $0\le i\le r-1$ and $u\in \Aut(S)$.

(c) The stable morphism represented by $\psi_i\o f_0$ corresponds to the self-extension $S\o V^{(i)}$ of $S$ under $\Ext^1(S,S)\simeq \uHom(\Omega S,S)$. 

\end{corollary}

\section{Carlson modules}
\label{sec:Carlson}

\subsection{Statement of the results}

Fix $r\in\mZ_{>1}$. We set $E=C_p^r$, and take
$$\ulambda = (\lambda_1,\lambda_2,\ldots,\lambda_r)\,\in\,\mA_f^{r}(\bk),$$ defining an embedding of $E$ into $SL_2(\bk)$, which we use to interpret $T_i$ as $E$-representations. In particular, we consider $V=T_1$ and
\begin{equation}\label{Sp2}
    S:=St_{r-1}=T_{p^{r-1}-1}=St_1\otimes St_1^{(1)}\otimes \cdots\otimes  St_1^{(r-2)}.
\end{equation}

\subsubsection{}
Recall, see for instance \cite{Benson-vector-bundles}*{\S 1.10}, that for a finite group $G$, $d\in\mN$ and $\zeta\in H^d(G,\bk)$, we have the corresponding Carlson module $L_\zeta$. This is defined via a short exact sequence
$$0\to L_\zeta\to \Omega^d\unit\xrightarrow{\hat{\zeta}}\unit\to 0,$$
where we used $\zeta\mapsto \hat{\zeta}$ for the isomorphism $H^d(G,\bk)\simeq \uHom_G(\Omega^d\unit,\unit)$.
We point out that, contrary to \cite{Benson-vector-bundles}, we do not make an exception for $\zeta=0$, meaning that $L_0\simeq \Omega^d \unit$.

The following theorem, which is the main result of this section, `verifies \Cref{conj-a} for Carlson modules'.

\begin{theorem}\label{ThmCarlson}
    For any $d\in\mN$ and $\zeta\in H^d(E,\bk)$, the tensor product
    $S\otimes L_\zeta$ is, up to projective summands, either isomorphic to $0$, to $S$, to
    $$T_{p^r-p^{r-1}+p^i-1}\;=\; \Omega(S\otimes V^{(i)}),\quad\text{ for some }0\le  i\le r-2,$$
    or, if $p>2$, isomorphic to
    $$T_{p^r-2p^{r-1}-1}\;=\; \Omega(S\otimes V^{(r-1)}),\quad\mbox{or}\quad T_{p^{r}-p^{r-1}-1} \;=\;\Omega(S) .$$ 
\end{theorem}

\begin{corollary} \label{cor:Carlson} Any Carlson module is contained in $\widetilde\cT$.
\end{corollary}

\begin{remark}
    Each option in \Cref{ThmCarlson} actually occurs, as follows from the proof. 
\end{remark}

\begin{example}\label{Exam:1}
We use the standard description of the cohomology rings of elementary $p$-groups, as recalled in \Cref{Carlson2} and \Cref{Carlsonodd} below.
   Using the isomorphism
        $$\chi:H^1(E,\bk)\;\xrightarrow{\sim}\; \mA^r, \quad \sum_{i=1}^r \mu_i y_i\mapsto (\mu_1,\mu_2,\ldots, \mu_r),$$
 for any non-zero $\zeta\in H^1(E,\bk)$, it follows from the definitions that, with notation for two-dimensional modules as in \ref{Sec2dim},
        $$ L_\zeta\;\simeq\; \Omega(V_{\chi(\zeta)})\quad\mbox{and hence}\quad S\otimes L_\zeta\simeq\Omega(S\otimes V_{\chi(\zeta)}).$$
        \Cref{ThmCarlson} thus also implies that every two-dimensional module is in $\widetilde{\cT}$.
\end{example}

The rest of the section is devoted to the proof. The proof will be split up into the separate cases $p=2$ and $p>2$, but we start with some generalities that will apply to both cases.

\subsection{Generalities on periodic modules}
Let $G$ be a finite group such that $p$ divides $|G|$.

\subsubsection{}\label{DefineOmega} For the remainder of this section, we {\em define} $\Omega:=\Omega^1$ as the functor on $\Rep G$ or $\Stab G$
$$\Omega\;:=\; -\otimes \Omega(\unit),$$
with $\Omega(\unit)$ defined as the kernel of the counit $\varepsilon: \bk G\to \bk$.

\subsubsection{}\label{central} For any $M\in \Rep G$, we have the algebra morphism
\begin{equation}\label{HomAlgMor}
    H^\ast(G,\bk)\;\simeq\; \Ext^\ast_G(\unit,\unit)\;\xrightarrow{M\otimes-}\; \Ext^\ast_G(M,M)\;\simeq\; \bigoplus_{i\ge 0}\uHom_G(\Omega^i M,M),
\end{equation}
where for convenience we assumed that $\End_G(M)\tto \uEnd_G(M)$ is an isomorphism. It is well-known that this algebra morphism is central in the `super sense', meaning that 
$$f (M\otimes h)\;=\; (-1)^{ab} (M\otimes h)f$$
for $f\in \Ext^a(M,M)$ and $h\in \Ext^b(\unit,\unit)$.

The algebra structure on the right-hand side in \eqref{HomAlgMor} is given by
$$fg\;:=\; f\circ \Omega^a(g),$$
for $f:\Omega^a\unit\to\unit$ and $g:\Omega^b\unit\to\unit$.

\subsubsection{}\label{tauperiod}

Now assume that $M$ is periodic with period $\tau\in\mZ_{>0}$. Then there exists 
$$\phi:\Omega^\tau M\xrightarrow{\sim}M  $$
in $\Stab G$. We assume further that $\phi$ is central as an element of $\Ext^\ast_G(M,M)$.
We can define the quotient algebra of $\Ext^\ast_G(M,M)$ by the ideal generated by the (non-homogeneous) central element 
$$\id_M-\phi\;\in \; \End_G(M)\oplus \uHom_G(\Omega^\tau M,M)\;\subset\;\Ext^\ast_G(M,M). $$
Since $\phi$ induces isomorphisms between $\uHom_G(\Omega^i M,M)$ and $\uHom_G(\Omega^{i+\tau} M,M),$
it follows that this quotient algebra of $\Ext^\ast_G(M,M)$ can naturally be identified with its subspace
$$\bigoplus_{i=0}^{\tau-1}\Ext^i_G(M,M)\;\simeq\; \bigoplus_{i= 0}^{\tau-1}\uHom_G(\Omega^i M,M).$$
With the inherited algebra structure on this space, \eqref{HomAlgMor} thus yields an algebra morphism
$$\Upsilon^{M}_{\phi}\;:\; H^\ast(G,\bk)\;\to\;\bigoplus_{i= 0}^{\tau-1}\uHom_G(\Omega^i M,M). $$

\begin{lemma}\label{LemPsi}\label{Propf}We use the assumptions and notation from \ref{tauperiod}.
    For $\zeta\in H^d(G,\bk)$, the object $M\otimes L_\zeta $ of $\Stab G$ depends only on $\Upsilon^M_\phi(\zeta)$. 
\end{lemma}
\begin{proof}
    By definition, and using $\Omega^dM=M\otimes \Omega^d\unit$, we have the exact triangle in $\Stab G$ 
    $$
    M\otimes L_\zeta \to \Omega^dM \xrightarrow{~M\otimes\hat{\zeta}~} M \to
    $$
where $M\otimes \hat{\zeta}$ is precisely the image of $\zeta$ under \eqref{HomAlgMor}. 
Moreover, the assignment that sends an element of $\sqcup_i \uHom(\Omega^iM,M)$ to the isomorphism class of the cocone in $\Stab G$ is invariant under the operation that sends $f:\Omega^iM\to M$ to $f\phi=f\circ \Omega^i(\phi)$, since $\Omega^i(\phi)$ is an isomorphism.
\end{proof}

\subsection{Even characteristic}\label{Carlson2}
Here we consider the case $p=\mathrm{char}(\bk)=2$. Then $St_1=V$, so in particular
\begin{equation}\label{Sp22}
    S:=St_{r-1}=T_{2^{r-1}-1}=V\otimes V^{(1)}\otimes \cdots\otimes  V^{(r-2)}.
\end{equation}

\subsubsection{}Recall from \cite{Benson-vector-bundles}*{\S 1.8} that 
$$H^\ast(E,\bk)\;\simeq\; \bk[y_1,\ldots , y_r]$$
as graded algebras, where $y_i$ has degree $1$, so that
$$\widehat{\cdot}\;:\;H^1(E,\bk)\xrightarrow{\sim}\uHom(\Omega^1\unit,\unit)=\Hom(\Omega^1\unit,\unit),\qquad \sum_{a=1}^r c_ay_a\mapsto \{X_i\mapsto c_i\},$$
with $\Omega^1\unit$ realised as the radical in $\bk C_2^r=\bk[X_1,\ldots,X_r]/(X_i^2)$.

\subsubsection{}\label{defphi2}For any non-zero $\zeta\in H^1(E,\bk)$, we have the epimorphism
$$S\otimes\hat{\zeta}\,:\, \Omega(S):=S\otimes\Omega^1\unit\to S$$
with kernel $S\otimes L_\zeta$. This morphism is an isomorphism in $\Stab E$ if and only if $S\otimes L_\zeta$ is projective, which, by \cite{Benson-vector-bundles}*{Proposition~1.10.3} and \Cref{Cor:Supp}, is equivalent to the condition $\sum_{i=1}^r \zeta_i \blambda_i$ being non-zero. Also by \Cref{Cor:Supp}, this condition is satisfied for
$$\zeta\;=\;\sum_{i=1}^r \lambda_i^{2^{r-1}}y_i.$$
Moreover, by construction, see \ref{central}, the resulting
\begin{equation}\label{phi2}\phi\,:\, \Omega(S)=S\otimes\Omega^1\unit\,\to\, S,\quad v\otimes X_i\mapsto \lambda_i^{p^{r-1}}v\end{equation}
yields an isomorphism in $\Stab E$ which is central in $\Ext^\ast(S,S)$.
In conclusion, we are in the situation of \ref{tauperiod}, with $\tau=1$ and algebra morphism
\begin{equation}\label{AlgM2}
    \Upsilon^{S}_{\phi}\;:\; H^\ast(E,\bk)\;\to\;\End_E(S).
\end{equation}

\begin{proof}[Proof of \Cref{ThmCarlson} for $p=2$]
By \Cref{cor:im-omega-1-1}(a) applied to the case $Q=V^{(r-1)}$, the image~$I$ of \eqref{AlgM2} restricted to
$$H^1(E,\bk)\;\to\; \End(S)$$
is spanned by the endomorphisms $\psi_i\in \End(S)$, for $0\le i\le r-1$. Since $\psi_0$ is the identity and $\psi_i\psi_j=0$ for $i,j>0$, it follows that this image $I$ is a subalgebra. The algebra $H^\ast(E,\bk)$ is generated in degree $1$ and hence the image of the algebra morphism \eqref{AlgM2} equals $I$.

By \Cref{Propf}, it is therefore sufficient to prove the theorem for $ L_\zeta\in H^1(E,\bk)$. Firstly, we have
$$S\otimes L_0\;=\; S\otimes\Omega^1\unit\;\simeq\; S.$$
By \Cref{cor:im-omega-1-1}(b) and (c) and \Cref{Exam:1}, the non-zero cases produce
$$S\otimes L_0\;\simeq\; S\otimes V^{(i)}\;\simeq\;T_{2^{r-1}+2^i-1},$$
which concludes the proof.
\end{proof}



\subsection{Odd characteristic}\label{Carlsonodd}

\subsubsection{}\label{HomOdd}
Here we consider the case $p=\mathrm{char}(\bk)>2$. Recall from \cite{Benson-vector-bundles}*{\S 1.8} that 
$$H^\ast(E,\bk)\;\simeq\; \bk[x_1,\ldots , x_r]\otimes \Lambda[y_1,\ldots , y_r]$$
as graded algebras, where $x_i$ is of degree $2$ and $y_i$ is of degree $1$.

We can construct an isomorphism $\Omega^2 S\to S$ similarly as for the case $p=2$ in \ref{defphi2}. Indeed, by \cite{Benson-vector-bundles}*{Proposition~1.10.3}\footnote{There is actually a minor mistake in the formulation of \cite{Benson-vector-bundles}*{Proposition~1.10.3} for $p$ odd. In the formula on the penultimate line $f^{(p)}(Y_1,\ldots, Y_r)$ should be replaced by $f(Y_1^p,\ldots, Y_r^p)$, see also \cite{Benson-vector-bundles}*{Corollary~1.10.2} or \cite{Carlson-coh-ring-mod}*{2.20, 4.5, or proof of 7.2}.}, an element $\zeta=\sum_a \zeta_a x_a\in H^2(E,\bk)$ induces an isomorphism $S\otimes \hat\zeta$ in $\Stab E$ if and only if $\sum_a\zeta_a\blambda_a^p\not=0$. Hence we can define
$$\phi: \;S\otimes \Omega^2\unit\xrightarrow{S\otimes \widehat{\zeta}}S,\qquad\mbox{for }\;\zeta=\sum_{a=1}^r\lambda_a^{p^{r}}x_a.$$
Again, since $H^{2\ast}(E,\bk)\to\Ext^\ast(S,S)$ is central, it follows that $\phi$ is central and we are in the situation of \ref{tauperiod} for $\tau=2$, with $\mZ/2$-graded algebra morphism
\begin{equation}\label{AlgMp}
    \Upsilon^{S}_{\phi}\;:\; H^\ast(E,\bk)\;\to\;\End_E(S)\oplus\uHom_E(\Omega S,S)
\end{equation}

\begin{lemma}\label{PrepH2}
    Under the isomorphism $\Ext^2_E(\unit,\unit)\simeq H^2(E,\bk)$, the exact sequence
$$0\to \unit\xrightarrow{1\mapsto e_p} St_1^{(l)}\xrightarrow{e_1\mapsto e_2} St_1^{(l)}\xrightarrow{e_1\mapsto 1} \unit\to 0,$$
for $0\le l<r$, is sent to $(\sum_a\lambda_a^{p^{l+1}}x_a)/(p-1)!$.
\end{lemma}
\begin{proof}
We prove the case $l=0$. The general case can be proved similarly, or simply reduced to $l=0$ by replacing $\ulambda$ with $\ulambda^{(l)}$.
    For $l=0$, we need to fill in the dashed arrows in the commutative diagram with exact rows
    $$\xymatrix{
    &\bk E^r\oplus \bk E^{\frac{1}{2}r(r-1)}\ar[r]\ar@{-->}[d]& \bk E^r\ar[r]\ar@{-->}[d]&\bk E\ar[r]\ar@{-->}[d]&\unit\ar[r]\ar@{=}[d]&0\\
    0\ar[r]&\unit\ar[r]&St_1\ar[r]& St_1\ar[r] & \unit\ar[r]&0.
    }$$
In the upper row, $\bk E\to\unit$ is the counit, the morphism $(\bk E)^r\to\bk E$ sends the $r$ generators $1\in \bk E$ to $\{X_i\mid 1\le i\le n\}$. For the left-most arrow on the top row, the first $r$ copies of $\bk E$ send $1$ to $X_i^{p-1}$ in the corresponding copy of $\bk E$ in the target. The remaining $r(r-1)/2$ copies of $\bk E$, labelled by $1\le i<j\le r$, send their generator to the difference of $X_i$ in the $j$th copy with $X_j$ in the $i$th copy. The right term in the upper arrow is then naturally labelled by the $r$ elements $x_i\in H^2(E,\bk)$ and the $r(r-1)/2$ elements $y_iy_j\in H^2(E,\bk)$, for $i<j$.

Using the rules
$$X_i e_a\;=\; \sum_{b>a}\frac{\lambda_i^{b-a}}{(b-a)!}e_b,\quad \mbox{for}\; 1\le i\le n\mbox{ and }1\le a\le p,$$
then allows one to verify that the left-most dashed arrow corresponds to $(\sum_i\lambda_i^p x_i)/(p-1)!$.
\end{proof}

\begin{lemma}\label{ImH2}
    The image of
    $$\Span\{x_1,\ldots, x_r\}\subset H^2(E,\bk)\;\to\; \uEnd_E(S)$$
    under $\Upsilon^S_\phi$ is $\bk \id_S$.
\end{lemma}
\begin{proof} By \Cref{Lem:Moore}, the determinant of $\mathbb{M}(\ulambda)$, or equivalently, that of $\mathbb{M}(\ulambda^{(1)})$ is non-zero. It follows that the span of the elements $\sum_{a=1}^r \lambda_a^{p^{i+1}}x_a$, for $0\le i<$r equals the span of the $x_a$.
    It follows from \Cref{PrepH2} and our choice of $\phi$ that
    $$\theta_i\;:=\;\Upsilon^S_\phi\left(\sum_{a=1}^r \lambda_a^{p^{i+1}}x_a\right)/(p-1)!$$
    fits into a commutative diagram
    $$\xymatrix{
0\ar[r]& S\ar[r]^-{v\mapsto v\otimes e_p}\ar[d]_{\theta_i}&S\otimes St_1^{(r-1)}\ar[d]\ar[rr]^-{v\otimes e_j\mapsto v\otimes e_{j+1}}&&S\otimes St_1^{(r-1)}\ar[r]^-{v\otimes e_1\mapsto v}\ar[d]& S\ar[r]\ar@{=}[d]&0\\
0\ar[r]& S\ar[r]_-{v\mapsto v\otimes e_p}&S\otimes St_1^{(i)}\ar[rr]_{v\otimes e_j\mapsto v\otimes e_{j+1}}&&S\otimes St_1^{(i)}\ar[r]_-{v\otimes e_1\mapsto v}& S\ar[r]&0.
}$$
Furthermore, any commutative diagram of the above form must have $\theta_i$ for its left arrow. In particular, $\theta_{r-1}=\id_S$.

Now consider $i<r-1$. We can choose the right morphism from $S\otimes St_1^{(r-1)}$ to $S\otimes St_1^{(i)}$ to send $v\otimes e_1$ to $v\otimes e_1$. With $\xi$ from \ref{defxi} it follows from \Cref{Powxi}(1) that the diagonal morphism from $S\otimes St_1^{(r-1)}$ to $S\otimes St_1^{(i)}$ then sends $v\otimes e_1$ to $\xi(v\otimes e_1)$. Applying \Cref{Powxi}(2) twice shows that this element is contained in
$$S\otimes \rad St_1^{(i)}\;\cap\; St_1\otimes \cdots \otimes St_1^{(i-1)}\otimes \rad St_1^{(i)}\otimes St_1^{(i+1)}\otimes\cdots St_1^{(r-2)}\otimes St_1^{(i)}. $$
Consequently, the left morphism from $S\otimes St_1^{(r-1)}$ to $S\otimes St_1^{(i)}$ can be chosen to take values in
$$St_1\otimes \cdots \otimes St_1^{(i-1)}\otimes \rad St_1^{(i)}\otimes St_1^{(i+1)}\otimes\cdots St_1^{(r-2)}\otimes St_1^{(i)}.$$
By \Cref{Powxi}(1) it then follows that the composition of $\theta_i$ with the inclusion $S\subset S\otimes St_1^{(i)}$ takes values in
\begin{eqnarray*}
    &&\xi^{p-1}\left(St_1\otimes \cdots \otimes St_1^{(i-1)}\otimes \rad St_1^{(i)}\otimes St_1^{(i+1)}\otimes\cdots St_1^{(r-2)}\otimes St_1^{(i)}\right)\\
    &\subset& St_1\otimes \cdots \otimes St_1^{(i-1)}\otimes \rad^p St_1^{(i)}\otimes St_1^{(i+1)}\otimes\cdots St_1^{(r-2)}\otimes St_1^{(i)}\;=\;0,
\end{eqnarray*}
where the inclusion follows from \Cref{Powxi}(2) applied to the non-standard choice of factors $S$ and $St_1^{(i)}$ in $S\otimes St_1^{(i)}$. Hence, we find $\theta_i=0$, which concludes the proof.
\end{proof}

\begin{proof}[Proof of \Cref{ThmCarlson} for $p>2$]
    First we show that the image of
    $$H^1(E,\bk)^{\otimes 2}\;\to\; \uHom_E(\Omega S,S)^{\otimes 2}\;\to\; \uEnd_E(S)$$
    is zero, where the second map is multiplication inside the algebra structure on
    $$\uEnd_E(S)\oplus\uHom_E(\Omega S,S)$$
    from \ref{tauperiod}.
    Indeed, let $f\in \uHom_E(\Omega S,S)$ correspond to $1\otimes f_0$ as in \Cref{cor:im-omega-1-1}, for $Q=St_1^{(r-1)}$, in particular, it is in the image of $H^1(E,\bk)\to\uHom(\Omega S,S)$. Then it follows that the image of $H^1(E,k)$ in $\uHom(\Omega S,S)$ consists of elements $af$ with $a\in \End(S)$. But inside the algebra $\Ext^\ast_E(S,S)$ (so certainly inside the above quotient) we have, for arbitrary $a,b\in\uEnd_E(S)$, we have
    $$(af)(bf)\;=\;ab f^2\;=\;0,$$
    as elements of $H^1(E,\bk)$ square to zero.

    The previous paragraph, together with \Cref{ImH2}, shows that the image of $H^2(E,\bk)$ under~$\Upsilon_\phi^S$ in \eqref{AlgMp} is the span of the identity. It thus follows that the image of $H^1(E,\bk)\oplus H^2(E,\bk)$ under~$\Upsilon_\phi^S$ is a $\mZ/2$-graded subalgebra. Indeed, it consists of the direct sum of the space spanned by the identity and a space that squares to zero, by the above paragraph.

    By \Cref{LemPsi}, and the fact that $H^\ast(E,\bk)$ is generated in degrees 1 and 2 (and the fact that morphism~\eqref{AlgMp} is $\mZ/2$-graded), it thus suffices to prove the theorem for $d=1$ and $d=2$. It follows from the above description of the image of $H^2(E,\bk)$ in $\End(S)$ and \Cref{LemPsi} that, for any $\zeta\in H^2(E,\bk)$, the module $S\otimes L_\zeta$ is, up to projective summands, $0$ or isomorphic to $S$. It follows from \Cref{cor:im-omega-1-1}(2) and \Cref{Exam:1} that for $\zeta\in H^1(E,\bk)$, the module
    $S\otimes L_\zeta$ is, up to projective summands,
    $$\Omega(S\otimes V^{(i)})\;\simeq\; \Omega(T_{p^{r-1}-1+p^i})\;\simeq\;
    \begin{cases} 
    T_{p^{r}-p^{r-1}+p^i-1},&\text{for } 0\le i<r-1,\\
    T_{p^r-2p^{r-1}-1},&\text{for } i=r-1,
    \end{cases}$$
    where we used \Cref{prop:Omega-Ti}, or alternatively $S\otimes L_\zeta\simeq\Omega S$.
\end{proof}

\section{Modules of Loewy length two} \label{sec:Loewy}

In this section, we will show that, for general $p$ and $r$, tensor products $S\o X$, for modules $X$ of Loewy length $2$, split into modules of the form $S\o Y$, for $Y$ of dimension $1$ or $2$. It will then follow that $S\o X$ is always in $\cT$ (and hence, even in $\cT_r$ by \Cref{rem:T-tilde}(4)).

\subsection{Simplifying extensions} \label{sec::simplify-ext}
We start with a very general observation. Let $\cA$ be a Frobenius category, let $\stHom$ and $\stEnd$ denote hom-spaces in its stable category $\underline{\cA}$. For $X,Y\in\cA$ and $f\in\stHom(\Omega X,Y)$, we denote by $E_f$ the extension of $Y$ by $X$ determined by $f$.

\begin{lemma} \label{lem::extensions-posets} Consider two finite sets of objects $(X_i)_{i\in I}$ and $(Y_j)_{j\in J}$ and morphisms $f_{ij}\in\stHom(\Omega X_i,Y_j)$ in $\underline{\cA}$, for all $i\in I,j\in J$.
Assume $k\in I,\ell\in J$ are fixed indices and for all $k\neq i\in I$ and $\ell\neq j\in J$, there are $a_i\in\stHom(\Omega X_i,\Omega X_k)$ and $a'_j\in\stHom(Y_\ell, Y_j)$ such that
$$
f_{i\ell} = f_{k\ell} a_i
\quad\text{and}\quad
f_{kj}= a'_j f_{k\ell} .
$$
Then $E_f$, the extension of $\bigoplus_j Y_j$ by $\bigoplus_i X_i$ determined by $f=(f_{ij})_{ij}$, is isomorphic to a direct sum $E_{f_{kl}}\oplus E_{\hat f}$, where $E_{\hat f}$ is the extension of $\bigoplus_{j\neq\ell} Y_j$ by $\bigoplus_{i\neq k} X_i$ determined by morphisms $\hat f=(\hat f_{ij})_{i\neq k,j\neq \ell}$, where
$$
\hat f_{ij} 
= f_{ij} - f_{kj} a_i 
= f_{ij} - a'_j f_{k\ell} a_i
= f_{ij} - a'_j f_{i\ell}
\quad\text{for all }i\neq k,j\neq\ell .
$$
\end{lemma}

\begin{proof} The statement is a version of Gaussian elimination for the matrix $(f_{ij})_{i,j}$.

Set $a_k:=0\in \stEnd(\Omega X_k)$ and $a'_\ell:=0\in\stEnd(Y_\ell)$. Let $\sigma$ be the automorphism of $\bigoplus_i \Omega X_i$ uniquely determined by its restrictions $\sigma|_{X_i}=\id_{\Omega X_i}-a_i$, let $\tau$ be the automorphism of $\bigoplus_j Y_j$ uniquely determined by its corestrictions $\tau|^{Y_j}=\id_{Y_j}-a'_j$. Then $\tau f \sigma$ is a direct sum of the morphism $f_{k\ell}$ and the morphism $\hat f:\bigoplus_{i\neq k}\Omega X_i\to\bigoplus_{j\neq \ell} Y_j$ specified above.
\end{proof}

\begin{corollary} \label{cor::linear-poset-ext}
Consider $X\in\cA$. Assume $A\subset\stEnd(X)$ and $A'\subset\stEnd(\Omega X)$ are subalgebras and $M\subset\stHom(\Omega X,X)$ is an $(A,A')$-subbimodule.  Assume $(g_k)_{1\le k\le t}$ is a finite set of elements in $M$ such that

(a) for any non-zero $g\in M$, there is an index $k$ and automorphisms $u\in A^\times$, $u'\in (A')^\times$ such that $g=u g_k=g_k u'$, and

(b) $g_k\in A g_\ell$ and $g_k\in g_\ell A'$ for all $1\le k\le \ell\le t$.

Assume $f=(f_{ij})_{i,j}$ is a $m\times n$-matrix with entries in $M$. Then $E_f$ is isomorphic to a direct sum of objects from the set $\{X\}\cup\{E_{g_k}\}_k$.
\end{corollary} 

\begin{proof} This follows from repeated applications of \Cref{lem::extensions-posets}.

Indeed, each non-zero $f_{ij}$ equals $g_k$ for some $k$ up to units in $A$ and $A'$ by (a). Let $k_0$ be the maximal $k$ appearing, let $i_0,j_0$ be such that $f_{i_0,j_0}$ corresponds to $g_{k_0}$. Then by \Cref{lem::extensions-posets} using (b), $E_f$ has a direct summand isomorphic to $E_{m_{k_0}}$, and one isomorphic to $\hat f$, where $\hat f$ is the $(m-1)\times(n-1)$-matrix obtained from $f$ via Gaussian elimination and removal of the $i_0$-th row and $j_0$-th column, using elements of $A$ and $A'$ as ``scalars''.

The assumptions ensure that $\hat f$ will still be of the form required to repeat the process as long as $\hat f$ has at least one row and column.
\end{proof}

\subsection{Set-up} 
We now turn to the category $\Rep C_p^r$, for an arbitrary prime $p$ and some $r\ge1$. As before, we fix an embedding $C_p^r\hookrightarrow SL_2$. We denote by $(T_i)_{0\le i\le p^r-1}$ the restrictions of the $SL_2$-tilting objects in $\Rep C_p^r$, and by $(St_j=T_{p^j-1})_{1\le j\le r}$ the restricted Steinberg modules. Also as before, we set $V:=T_1$ and $S:=St_{r-1}$.

\begin{lemma} \label{lem::Omega-stD} $\Omega$ restricts to the endofunctor $-\o T_{p-2}^{(r-1)}$ on the full subcategory of $\Stab C_p^r$ on $S$-projective objects.
\end{lemma}

\begin{proof} $S$, and hence any non-projective $S$-projective object $X$, has a one-point support consisting of a point not fixed by the Frobenius automorphism by \Cref{Cor:Supp}. Hence, $T_{p-1}^{(r-1)}\o S\o X$ is projective, and the assertion follows, as $T_{p-2}^{(r-1)}\simeq\ker(T_{p-1}^{(r-1)}\to\one)$.
\end{proof}

In the following, we consider $\stHom(\Omega S, S)$ as an $\stEnd(S)$-bimodule, using $\Omega$ for the right action.

\begin{lemma} \label{lem::Hom-Omega-S-S} For any non-zero $C_p^r$-morphism $f_0:T_{p-2}^{(r-1)}\to\one$, the map 
$$
\stEnd(S)\to\stHom(\Omega S,S)\simeq\stHom(S\o T_{p-2}^{(r-1)} ,S), \quad
f\mapsto f\o f_0,
$$
is an isomorphism of $\stEnd(S)$-bimodules, using the identifications
$$
\stEnd(S)\simeq\stEnd(\Omega S)\simeq\stEnd(S\o T_{p-2}^{(r-1)}), \quad
f\mapsto \Omega f\mapsto f\o T_{p-2}^{(r-1)} .
$$
\end{lemma}

\begin{proof}
    The map is a well-defined morphism of bimodules. $S\o\Omega S=T_{p^{r-1}-1}\o T_{p^r-p^{r-1}-1}$ has no projective summands, as follows directly from considering the tensor product of the corresponding $SL_2$-tilting modules, which does not have a submodule isomorphic to $T_{p^r-1}$, as the corresponding highest weight does not appear. This means $\stHom(\Omega S,S)\simeq\Hom(\Omega S,S)$, which means the described bimodule morphism is injective. 
    
    As $\Omega S$ is cyclic, 
    $$\dim\stHom(\Omega S,S)\le\dim S= p^{r-1},$$ which is exactly the dimension of $\stEnd(S)$ by \Cref{LemmaLambda}.
\end{proof}

For the rest of the section, we fix a choice of $f_0:T_{p-2}^{(r-1)}\tto\one$.

\subsection{Groups of rank two} 
We first consider $r=2$. While the implications (like \Cref{prop::self-ext-S}) will also follow from the more general results in \Cref{sec::LL2-r-arbitrary}, they are somewhat more elementary, as they do not rely on the ideas and results of \Cref{sec::subcategories}.

\begin{lemma} \label{lem::ext-maps-r2} Assume $r=2$. Then there are $(\phi_i)_{1\le i\le p}\in\stEnd(S)$ such that

(a) for any non-zero element $g\in\stHom(\Omega S,S)$, there is an $1\le i\le p$ and automorphisms $u,u'\in\stAut(S)$ such that $g=u \phi_i \o f_0=\phi_i u'\o f_0$, and

(b) $\phi_i\in \phi_j\stEnd(\Omega S)$ and $\phi_i\in\stEnd(S)\phi_j$ for all $1\le j\le i\le p$. 
\end{lemma}

\begin{proof} Set $A:=\stEnd(S)\simeq\End(S)\simeq \kk[z]/z^p$, see \Cref{LemmaLambda}. By \Cref{lem::Hom-Omega-S-S}, $M:=\stHom(\Omega S,S)$ is isomorphic as an $A$-bimodule to the regular $A$-bimodule. 

As $z\in A$ is nilpotent, any element of $A$ is a power of $z$ up to multiplication by a unit. The assertions follow with $\phi_i:=z^{i-1}$.
\end{proof}

\begin{proposition} \label{prop::self-ext-S} Assume $r=2$. For all $a,b\ge0$, any extension of $S^a$ by $S^b$ is in $\langle S=T_{p-1},T_p,\dots, T_{2p-1}\rangle_\oplus\subset\cT$. 
\end{proposition}

\begin{proof} The extensions are parametrised by $\stHom((\Omega S)^a,S^b)\simeq \stHom(\Omega S,S)^{a\times b}$.

It follows from \Cref{lem::ext-maps-r2} with \Cref{cor::linear-poset-ext} that any extension as above is a direct sum of self-extensions of $S$ corresponding to the $p$ maps $\phi_i$, and copies of $S$ itself.

By \Cref{lem::ext-maps-r2} there are at most $p$ isomorphism classes of non-split self-extensions of $S$. The conclusion follows from \Cref{cor:self-ext1} which realises $p$ such extensions.
\end{proof}

It follows that \Cref{conj-a} is true for $r=2$ when restricted to modules of Loewy length at most $2$:
\begin{corollary} \label{cor:ll-2}
    Any module of Loewy length at most $2$ lies in $\widetilde\cT$.
\end{corollary}

\begin{proof}
    $S\o X$, for $X$ of Loewy length $2$, is of the form studied in \Cref{prop::self-ext-S}.
\end{proof}

\begin{corollary} \label{cor:extensions-p2} For $p=2$ and $r=2$, the tensor product $S\o X$ is in $\langle T_1,T_2,T_3\rangle_\oplus=\cT_r\subset\cT$ for all $X\in\Rep C_2^2$.
\end{corollary}

\begin{proof} It suffices to consider non-projective indecomposable objects $X$. As the projective indecomposable $C_2^2$-module has Loewy length $3$, such $X$ have Loewy length at most $2$.
\end{proof}

The above proves \Cref{conj-a} for $p=r=2$. A direct proof of this case using the classification of indecomposable $C_2^2$-modules was given in \cite{CF}*{Remark~5.4.4}. \Cref{cor:extensions-p2} is the simplest proof of \Cref{conj-b} in this special case without invoking the classification of indecomposable $C_2^2$-modules that we are aware of.

\begin{corollary}\label{CorSelfExt} For any $p$ and $r=2$, and for $a,b\ge0$, any extension of $(\Omega S)^a$ by $(\Omega S)^b$ is in $\langle  T_{p^2-2p-1},\Omega S=T_{p^2-p-1},T_{p^2-p},\dots,T_{p^2-2}\rangle_\oplus$. 
\end{corollary}

\begin{proof} By taking Heller shifts and identifying the obtained modules using \Cref{prop:Omega-Ti}.  
\end{proof}

\begin{example} \label{expl::self-ext-S} For $p=3$, $r=2$, the non-split self-extensions of $S$ are exactly the modules $T_3,T_4,T_5$ by \Cref{prop::self-ext-S} and its proof.
\end{example}

\subsection{Groups of arbitrary rank} \label{sec::LL2-r-arbitrary}

For $r>2$, \Cref{prop::self-ext-S} will no longer be true, see \Cref{rem-ex}, but an analogue of \Cref{cor:ll-2} will still hold.
To establish this, we restrict to ``$S$-projective'' extensions. Consider the map $$
\omega_{\one,\one}=S\o-:\stHom(\Omega\one,\one)\to\stHom(\Omega S,S).
$$
We summarize some result from \Cref{sec::subcategories}, yielding an ``$S$-projective'' analogue of \Cref{lem::ext-maps-r2}.

\begin{lemma} \label{lem::ext-maps-any-r} There is a subalgebra $A\subset\stEnd(S)$ and elements $(\psi_i)_{1\le i\le r}\in A$ such that $M:=A\o f_0=\im\omega_{\one,\one}$ and

(a) for any non-zero element $g\in M$, there is an $1\le i\le r$ and automorphisms $u,u'\in A^\times$ such that $g=u \psi_i \o f_0=\psi_i u'\o f_0$, 

(b) $\psi_i\in \psi_j A$ and $\psi_i\in A\psi_j$ for all $1\le j\le i\le r$,

\end{lemma}

\begin{proof} This follows immediately from \Cref{cor:im-omega-1-1}.   
\end{proof}

\begin{proposition} \label{prop::self-ext-S-any-r} Let $X$ be any module of Loewy length $2$. Then $S\o X$ is in $\langle T_{p^{r-1}-1}=S, T_{p^{r-1}}\simeq S\o V, T_{p^{r-1}+p-1}\simeq S\o V^{(1)},\dots, T_{p^{r-1}+p^{r-1}-1}\simeq S\o V^{(r-1)} \rangle_\oplus \subset \cT$.
\end{proposition}

\begin{proof} The object $S\o X$ can be viewed as an extension of $S^a$ by $S^b$, for $a=\dim(X)-\dim\soc(X)$ and $b=\dim\soc(X)$, corresponding to a map in $(\im\omega_{\one,\one})^{a\times b}\subset \stHom(\Omega S,S)^{a\times b}$.

It follows from \Cref{lem::ext-maps-any-r} with \Cref{cor::linear-poset-ext} that any extension as above is a direct sum of self-extensions of $S$ corresponding to the maps $\psi_i$ and copies of $S$ itself.

The self-extensions of $S$ corresponding to the maps $(\psi_i)_{1\le i\le r}$ are identified as the modules $(S\o V^{(j)})_{1\le j\le r}$ in \Cref{cor:im-omega-1-1}.
\end{proof}

\begin{corollary} \label{cor:ll-2-any-r}
    Any module of Loewy length at most $2$ lies in $\widetilde\cT$.
\end{corollary}

\begin{example}\label{ex-moreex} Consider $p=3,r=2$. In view of \Cref{expl::self-ext-S}, note that out of the three non-split self-extensions of $S$, only only $T_3\cong S\o V$ and $T_5\simeq S\o V^{(1)}$ can appear in a tensor product of the form $S\o X$, for $X$ of Loewy length $2$.
\end{example}


\section{Cyclic \texorpdfstring{$S$}{S}-projective modules}
\label{sec:cyclic}

In this section, we take a slightly different approach to the previous sections. Before we proved that for specific $E$-modules $X$, the module $S\otimes X$ belongs to $\cT$ as conjectured. In the current section, we look at specific (namely cyclic) modules $Y$ that can appear in summands of modules $S\otimes X$, for arbitrary $X$, and show that they must be in $\cT$ as conjectured.

As before, we fix a prime $p$ and some $r\ge1$. The main results will be for $r=2$. We still have $E:=C_p^r$. The sets $(T_i)_{0\le i\le p^r-1}$ and $(\Delta_i)_{0\le i\le p^r-1}$ denote the restricted tilting and standard modules, and $St_i:=T_{p^i-1}$ is the restricted Steinberg modules for $0\le i\le r$. We set $S:=St_{r-1}$.

\subsection{Generalities}

Recall from \Cref{DefWProj} the notions of $S$-projective modules and $S$-split morphisms.
Also as before, we denote by $(g_i)_{1\le i\le r}$ the generators of $E=C_p^r$.
We start by establishing some general properties of $S$-projective modules (note that this is in accordance with \Cref{lem:consequences}(5)).

\begin{lemma} \label{lem:dim-periodic} For any $p,r$, let $X$ be an indecomposable non-zero non-projective $S$-projective $E$-module. 

(a) The restriction of $X$ to the subgroup $\langle g_1,\dots,g_{r-1}\rangle\cong C_p^{r-1}$ is projective.

(b) Set $d:=\dim\top X$, then $\dim X\in p^{r-1} \{d,d+1,\dots,(p-1)d\}$.

(c) If $p=2$, then $\dim X=2^{r-1} \dim\top X$.
\end{lemma}

\begin{proof} (a) Such an $X$ has the same one-point support as $S$, as computed in \Cref{Cor:Supp}. This specific point does not lie in the space of points corresponding to the cyclic shifted subgroups of $G:=\langle g_1,\dots,g_{r-1}\rangle$. Hence, the restrictions to those cyclic shifted subgroups are projective, and the restriction $X'$ of $X$ to $G$ must be projective.

(b) It follows from (a) that $\dim X=\dim X'$ is a multiple of $p^{r-1}$. As $\dim\top X'\ge\dim\top X=d$, we even have $\dim X=\dim X'\ge d p^{r-1}$. We also have
$$
\dim X + \dim \Omega X = \dim P^d = p^r d .
$$
Both $X$ and $\Omega X$ are $S$-projective and non-zero. In particular, $\Omega X$ is projective when restricted to $G$, too, and $\dim\soc (\Omega X)=d$, as otherwise $X$ would have a non-zero projective direct summand. Thus, the restriction of $\Omega X$ to $G$ has a socle of dimension at least $d$, so $\dim\Omega X\ge dp^{r-1}$, as well. 
This establishes the asserted range for $\dim X$. 

(c) follows immediately from (b).
\end{proof}

\begin{lemma} Assume $S\o Y$ is projective. Then $Y$ cannot appear as a quotient or subobject in any indecomposable non-projective $S$-projective $E$-module $X$.
\end{lemma}

\begin{proof} Let $\pi:Q\tto Y$ be a projective cover. As $S\o Y$ is projective, $\pi$ is $S$-split. As $\pi$ is a projective cover, the tops of $Q$ and $Y$ have the same dimension. Due to $S$-projectivity of $X$, any epimorphism $X\tto Y$ factors via some $f:X\to Q$. As any morphism is surjective if and only if it is surjective onto the top, $f$ must be surjective.

This shows $Q$ appears as a quotient, hence as a direct summand, in any $S$-projective object that has $Y$ as a quotient. The argument for the statement on subobjects is similar, using that $S$-projectivity is the same as $S$-injectivity.
\end{proof}

\begin{corollary} \label{cor:quotients} $T_{p-1}^{(r-1)}$ cannot be a quotient or subobject in any indecomposable non-projective $S$-projective module.
\end{corollary}

\subsection{Classification of cyclic \texorpdfstring{$S$}{S}-projectives for \texorpdfstring{$r=2$}{r=2}}

Assume $p$ arbitrary and $r=2$.
The following is the main result of this section, and proves a special case of the formulation of our main conjecture in \Cref{conj-c}.
\begin{theorem} \label{thm:cyclic-S-proj}
    If $r=2$, then the cyclic $S$-projective $C_p^r$-modules are in $\cT_r$.
\end{theorem}
The theorem will be proved as the more specific \Cref{prop-cyc} below.

We set $\Delta_{-1}:=T_{-1}:=0$.
Recall from \Cref{Lem:UseSL2} that for all $0\le i\le j$, there is an $E$-module epimorphism
$$\Delta_j\tto\Delta_i$$
whose kernel is the subspace of $SL_2$-weight vectors of weight strictly below $j-2i$. 

\begin{lemma} \label{lem:cyc-S-proj-r2}
(a) $\dim \Delta_i=i+1$ and $\Delta_i$ is cyclic for all $-1\le i\le p^2-1$. $\Delta_{\ell p-1}\cong T_{\ell p-1}$ for all $0\le \ell\le p$.

(b) For all $0< \ell\le p-1$, $\Delta_{\ell p}$ is the unique non-split extension of $\one$ by $T_{\ell p-1}$.

(c) For all $0\le k,\ell\le p-2$, there is an $S$-split epimorphism 
$$ T_{2p-2-k+\ell p} \tto \Delta_{k+\ell p} .
$$
\end{lemma}

\begin{proof}
(a) Cyclicity follows as $\Delta_{p^2-1}\cong T_{p^2-1}$ is the indecomposable projective $E$-module and $\Delta_i$ is a quotient of the latter for all $0\le i\le p^2-1$. The rest follows from well-known facts about standard modules for $SL_2$.

(b) By \Cref{prop:Omega-Ti}, $\Omega T_{\ell p-1}\cong T_{(p-\ell)p-1}$ is cyclic. Hence, $\dim\Hom(\Omega T_{\ell p-1},\one)=1$ and there is a unique such extension. $\Delta_{\ell p}$ is such an extension, as $\Delta_{\ell p-1}\cong T_{\ell p-1}$ is a quotient, and the difference of the dimensions is $1$.

(c) Recall that for any simple $SL_2$-module $L_i$ of highest weight $i\le p^r-1$, $L_i\o St_{r-1}$ is tilting as an $SL_2$-module (\cite{BEO}*{Lemma~3.3}, \cite{AbEnv}*{Lemma~4.3.4}). As extensions between tiltings vanish, this implies that any $SL_2$-epimorphism between modules with composition factors in this range splits upon tensoring with $S=St_{r-1}$ (\cite{BEO}*{Proposition~3.2}).

It is known that $T_{2p-2-k}$ has a composition series of the form $[L_k,L_{2p-2-k},L_k]$, see for instance \cite{DT}*{Lemma~1.1(b)}. We compute
$$L_{2p-2-k}\otimes L_{\ell}^{(1)}=L_{p-2-k}\otimes (L_1\otimes L_{\ell})^{(1)}=L_{2p-2-k+\ell p}\oplus L_{-2-k+\ell p}.$$
Hence, $T_{2p-2-k+\ell p}\cong T_{2p-2-k}\o T_\ell^{(1)}$ has a composition series $[L_{k+\ell p},L_{2p-2-k+\ell p}, L_{-2-k+\ell p},L_{k+\ell p}]$. This shows $T_{2p-2-k+\ell p}$ has the desired composition factors and the desired $SL_2$-epimorphism to $\Delta_{k+\ell p}$.
\end{proof}

Recall that we have fixed generators denoted by the symbol $1$ and endomorphisms denoted by the symbol $z$ with one-dimensional kernels for all uniserial $E$-modules, and that we denote the elements $z^k(1)$ also by $z^k$, for all $k\ge0$. In particular, this applies to $T_i$ for $0\le i\le p-1$ and its Frobenius twists.

\begin{lemma} \label{lem:realizations}
(a) For $0\le k,\ell\le p-1$, we have an isomorphism (of $E$-modules)
$$
\Delta_{k+\ell p} 
\cong \frac{T_{p-1}\o T_\ell^{(1)}}{\langle z^{k+1}\o z^\ell\rangle}
. $$

 (b) \label{lem:T-prime} For $0\le k,\ell\le p-2$, we have an isomorphism (of $E$-modules)
$$
T_{2p-2-k+\ell p} 
\cong \frac{(T_{p-1}\o T_{\ell+1}^{(1)}) \oplus (T_{p-1}\o T_\ell^{(1)})}{\langle(1\o z^{\ell+1}, -z^{k+1}\o z^\ell)\rangle}
.
$$

(c) For all $0\le k,\ell\le p-2$, any cyclic submodule of $T_{2p-2-k+\ell p}$ not contained in the radical has $\Delta_{k+1+\ell p}$ as a quotient.

\end{lemma}

\begin{proof}
(a) It is well-known that $\Delta_{p-1+\ell p}\cong T_{p-1+\ell p}\simeq T_{p-1}\otimes T_\ell^{(1)}$. By \Cref{Lem:UseSL2}, $\Delta_{k+\ell p}$ is the quotient of the latter with respect to the subspace of $SL_2$-weight vectors of weight strictly below $p-1+\ell p-2(k+\ell p)$. Those weight vectors are exactly $\{z^{k'}\o z^\ell\}_{k+1\le k'\le p-1}$, and the submodule they form is generated by $z^{k+1}\o z^\ell$.

(b) Fix $\xi\in\kk E$ such that $\xi|_{T_{p-1}\o T_{p-1}^{(1)}}=1\o z$, using that $T_{p-1}\o T_{p-1}^{(1)}\cong\kk E$, so in particular, $\xi$ acts trivally on $T_{p-1}$. Furthermore, as $\xi\in\rad(\kk E)$, the coproduct $\Delta$ of $\kk E$ satisfies 
\begin{equation} \label{eq:Delta-xi}    
\Delta(\xi) \in \xi\o1 + 1\o\xi+ \rad(\kk E)\o\rad(\kk E),
\end{equation}
which we will use to compute the action of $\xi$ on tensor products.

We claim that $\xi^2 T_{2p-2-k}=0$ and $\xi T_{2p-2-k}\cong L_{p-2-k}$. The first part of the claim follows using that, as an $E$-module, $T_{2p-2-k}$ is a self-extension of $T_{p-1}$ by \Cref{cor:self-ext1}, and $\xi|_{T_{p-1}}=0$. 

As can be derived from the proof of \Cref{lem:self-ext1}, we have the following description of $T_{2p-p-k}$: Set $M:=T_{p-1}\o T_{p-1-k}=L_{p-1}\o L_{p-1-k}$, set $v_1:=1\o 1$ and $v_2:=1\o z^{p-1-k}$ in $M$. Then $\langle v_1,v_2\rangle\cong T_{2p-2-k}$. Using \Cref{eq:Delta-xi}, we get $\xi(v_2)=0$, so $\xi T_{2p-2-k}=\xi\langle v_1\rangle$. The action of $\xi$ on $v_1=1\o 1$ is described in \Cref{Lemxi0} and \Cref{lem::xi-S-S}, which imply that $\xi(\Delta_{2p-2-k})\cong L_{p-2-k}$ as an $E$-module.

Now, since $T_{p-1}\o T_{p-1}^{(1)}$ is cyclic, $1\o z^{\ell+1}\in T_{p-1}\o T_{p-1}^{(1)}$ generates a submodule with corresponding quotient $T_{p-1}\o T_\ell^{(1)}\cong T_{p-1+\ell p}$. Hence, the submodule can be identified with $\Omega T_{p-1+\ell p}$, and the module on the RHS in the asserted equation is isomorphic to a module of the form 
$$T_{p-1+\ell p}\sqcup_{\Omega T_{p-1+\ell p}}\bk E,$$
corresponding to a morphism $\Omega T_{p-1+\ell p}\to T_{p-1+\ell p}$ with image $\im(z^{k+1}\o z^\ell)\subset T_{p-1}\o T_\ell^{(1)}$, or equivalently $L_{p-2-k}\subset T_{p-1+\ell p}$. As such a morphism is determined by the image of the generator, we find the module is the unique self-extension $Z$ of $T_{p-1+\ell p}$ satisfying $\xi^{\ell+1} Z\cong L_{p-2-k}$, where uniqueness follows form the observation that there is a unique embedding of $L_{p-2-k}$ into $T_{p-1+\ell p}$.

We claim that $T_{2p-2-k+\ell p}$ is a self-extension of $T_{p-1+\ell p}$. Indeed, by \Cref{cor:self-ext1}, $T_{2p-2-k}$ is a self-extension of $T_{p-1}$, and then we can take a tensor product with $T_{\ell}^{(1)}$. 
Using both parts of the above claim and \Cref{eq:Delta-xi}, we compute 
$$
\xi^{\ell+1} T_{2p-2-k+\ell p} 
\cong \Delta(\xi)^{\ell+1} (T_{2p-2-k} \o T_{\ell}^{(1)}) 
\cong \xi T_{2p-2-k} \cong L_{p-2-k} .
$$

This shows $T_{2p-2-k+\ell p}$ can also be identified as $Z$.
    
(c) 
By (b), $T':=T_{2p-2-k+\ell p}$ can be realized as
$$
T' \cong \frac{T_{p-1}\o T_{\ell+1}^{(1)}\oplus T_{p-1}\o T_\ell^{(1)}}{\langle (1\o z^{\ell+1},-z^{k+1}\o z^\ell)\rangle} .
$$
This means that the vectors 
$$
v_{ij}:=\ol{(z^i\o z^j,0)},
\qquad w_{ij}:=\ol{(0,z^i\o z^j)}
\qquad\text{for all } 0\le i\le p-1, 0\le j\le \ell, 
$$
form a basis of $T'$, that $\ol{(z^i\o z^{\ell+1},0)}=0$ for all $i\ge p-1-k$, and that
$$
K_1:=\langle v_{1,0}\rangle=\kk\{v_{ij}\}_{1\le i\le p-1, 0\le j\le \ell}+\kk\{w_{i\ell}\}_{k+2\le i\le p-1}
$$
$$
K_2:=\langle w_{0,0}\rangle=\kk\{w_{ij}\}_{0\le i\le p-1,0\le j\le \ell},
$$
are submodules. We compute using (a):
$$
T'/K_1\cong \frac{T_{\ell+1}^{(1)} \oplus T_{p-1}\o T_\ell^{(1)}}{\langle (z^{\ell+1},-z^{k+1}\o z^\ell),(0,z^{k+2}\o z^\ell) \rangle}
= T^{(1)}_{\ell+1} \sqcup_{\one} \Delta_{k+1+\ell p}
$$
and
$$
T'/K_2\cong T_{p-1}\o T_{\ell+1}^{(1)} \cong T_{p-1+\ell p} .
$$

Any cyclic submodule not in the radical is generated by a vector 
$$
v=\sum_{i,j} \alpha_{ij} v_{ij} + \sum_{i,j} \beta_{ij} w_{ij}
$$
with $\alpha_{0,0}\neq0$ or $\beta_{0,0}\neq0$.
If $\alpha_{0,0}\neq0$, then $\ol{v}\in T'/K_2$ generates $T_{p-1+\ell p}\cong\Delta_{p-1+\ell p}$, which has $\Delta_{k+1+\ell p}$ as a quotient.
If $\alpha_{0,0}=0$ and $\beta_{0,0}\neq 0$, then
$$
\ol{v}\in T'/K_1\cong T_{\ell+1}^{(1)}\sqcup_\one \Delta_{k+1+\ell p}
$$
is of the form $w+w'$, where $w$ is the image of a generator of $\Delta_{k+1+\ell p}$ and $w'$ is the image of some element in $\rad(T_{\ell+1}^{(1)})$. This means there is a nilpotent endomorphism of $T'/K_1$ sending $w$ to $w'$. Hence, using an automorphism of $T'/K_1$, we can conclude that the submodule generated by $\ol v$ is isomorphic to that generated by $w$, which is $\Delta_{k+1+\ell p}$.
\end{proof}

\begin{proposition}\label{prop-cyc} If $r=2$, then the non-zero cyclic $S$-projective $C_p^r$-modules are $\{T_{\ell p-1}: 1\le \ell\le p\}$.
\end{proposition}

\begin{proof} Let $X$ be a cyclic $S$-projective module.

If, for some $0\le \ell\le p-1$, the module $T_{\ell p-1}$ is a proper quotient of $X$, then also $\Delta_{\ell p}$ is a quotient of $X$. Indeed, the case $\ell=0$ is trivial, the case $\ell>0$ follows from
 \Cref{lem:cyc-S-proj-r2}(b) and the fact that $X$ is cyclic.

If for some $0\le \ell\le p-1$ and $0\le k\le p-2$, the module $\Delta_{k+\ell p}$ is a quotient of $X$, then also $\Delta_{k+1+\ell p}$ is a quotient of $X$,
by the lifting property of $X$ with respect to a morphism in \Cref{lem:cyc-S-proj-r2}(c) and by \Cref{lem:realizations}(c).

In conclusion, for $0\le\ell\le p-2$, if $T_{\ell p-1}$ is a proper quotient of $X$, then $\Delta_{p-1+\ell p}\cong T_{(\ell+1)p-1}$ is a quotient of $X$. Since $T_{-1}=0$ is a quotient of $X$, it follows that either $X$ is isomorphic to $T_{\ell p-1}$ for some $\ell\le p-1$ or that it has $T_{(p-1)p-1}$ as a proper quotient. In the latter case $X$ has to be projective by \Cref{lem:dim-periodic}, so $X\simeq T_{p^2-1}$ .
\end{proof}

\subsection*{Acknowledgements}
The authors thank Pavel Etingof and Raphaël Rouquier for useful discussions. KC was supported by ARC grants FT220100125 and DP210100251.

\bibliography{bib}

\begin{bibdiv}
\begin{biblist}

\bib{BCR-thick-subcategories}{article}{
      author={Benson, D.~J.},
      author={Carlson, Jon~F.},
      author={Rickard, Jeremy},
       title={Thick subcategories of the stable module category},
        date={1997},
        ISSN={0016-2736,1730-6329},
     journal={Fund. Math.},
      volume={153},
      number={1},
       pages={59\ndash 80},
         url={https://doi.org/10.4064/fm-153-1-59-80},
      review={\MR{1450996}},
}

\bib{BE}{article}{
      author={Benson, Dave},
      author={Etingof, Pavel},
       title={Symmetric tensor categories in characteristic 2},
        date={2019},
        ISSN={0001-8708,1090-2082},
     journal={Adv. Math.},
      volume={351},
       pages={967\ndash 999},
         url={https://doi.org/10.1016/j.aim.2019.05.020},
      review={\MR{3956002}},
}

\bib{Benson-mod-rep-th}{book}{
      author={Benson, David~J.},
       title={Modular representation theory},
      series={Lecture Notes in Mathematics},
   publisher={Springer-Verlag, Berlin},
        date={2006},
      volume={1081},
        ISBN={978-3-540-13389-6; 3-540-13389-5},
        note={New trends and methods, Second printing of the 1984 original},
      review={\MR{2267317}},
}

\bib{Benson-vector-bundles}{book}{
      author={Benson, David~J.},
       title={Representations of elementary abelian {$p$}-groups and vector bundles},
      series={Cambridge Tracts in Mathematics},
   publisher={Cambridge University Press, Cambridge},
        date={2017},
      volume={208},
        ISBN={978-1-107-17417-7},
         url={https://doi.org/10.1017/9781316795699},
      review={\MR{3585474}},
}

\bib{BEO}{article}{
      author={Benson, Dave},
      author={Etingof, Pavel},
      author={Ostrik, Victor},
       title={New incompressible symmetric tensor categories in positive characteristic},
        date={2023},
        ISSN={0012-7094,1547-7398},
     journal={Duke Math. J.},
      volume={172},
      number={1},
       pages={105\ndash 200},
         url={https://doi.org/10.1215/00127094-2022-0030},
      review={\MR{4533718}},
}

\bib{BIK-stratifying-modular}{article}{
      author={Benson, David~J.},
      author={Iyengar, Srikanth~B.},
      author={Krause, Henning},
       title={Stratifying modular representations of finite groups},
        date={2011},
        ISSN={0003-486X,1939-8980},
     journal={Ann. of Math. (2)},
      volume={174},
      number={3},
       pages={1643\ndash 1684},
         url={https://doi.org/10.4007/annals.2011.174.3.6},
      review={\MR{2846489}},
}

\bib{BIKP-stratification}{article}{
      author={Benson, Dave},
      author={Iyengar, Srikanth~B.},
      author={Krause, Henning},
      author={Pevtsova, Julia},
       title={Stratification for module categories of finite group schemes},
        date={2018},
        ISSN={0894-0347,1088-6834},
     journal={J. Amer. Math. Soc.},
      volume={31},
      number={1},
       pages={265\ndash 302},
         url={https://doi.org/10.1090/jams/887},
      review={\MR{3718455}},
}

\bib{Carlson-var-coh-ring-mod}{article}{
      author={Carlson, Jon~F.},
       title={The varieties and the cohomology ring of a module},
        date={1983},
        ISSN={0021-8693},
     journal={J. Algebra},
      volume={85},
      number={1},
       pages={104\ndash 143},
         url={https://doi.org/10.1016/0021-8693(83)90121-7},
      review={\MR{723070}},
}

\bib{Carlson-connected}{article}{
      author={Carlson, Jon~F.},
       title={The variety of an indecomposable module is connected},
        date={1984},
        ISSN={0020-9910,1432-1297},
     journal={Invent. Math.},
      volume={77},
      number={2},
       pages={291\ndash 299},
         url={https://doi.org/10.1007/BF01388448},
      review={\MR{752822}},
}

\bib{Carlson-coh-ring-mod}{article}{
      author={Carlson, Jon~F.},
       title={The cohomology ring of a module},
        date={1985},
        ISSN={0022-4049,1873-1376},
     journal={J. Pure Appl. Algebra},
      volume={36},
      number={2},
       pages={105\ndash 121},
         url={https://doi.org/10.1016/0022-4049(85)90064-7},
      review={\MR{787166}},
}

\bib{CEO}{article}{
      author={Coulembier, Kevin},
      author={Etingof, Pavel},
      author={Ostrik, Victor},
       title={On {F}robenius exact symmetric tensor categories},
        date={2023},
        ISSN={0003-486X,1939-8980},
     journal={Ann. of Math. (2)},
      volume={197},
      number={3},
       pages={1235\ndash 1279},
         url={https://doi.org/10.4007/annals.2023.197.3.5},
        note={With Appendix A by Alexander Kleshchev},
      review={\MR{4564264}},
}

\bib{CF}{article}{
      author={Coulembier, Kevin},
      author={Flake, Johannes},
       title={Towards higher frobenius functors for symmetric tensor categories},
        date={2024},
      eprint={2405.19506},
         url={https://arxiv.org/abs/2405.19506},
}

\bib{wild-periodic-modules}{article}{
      author={Carlson, Jon~F.},
      author={Jones, Alfredo},
       title={Wild categories of periodic modules},
        date={1988},
        ISSN={0019-2082,1945-6581},
     journal={Illinois J. Math.},
      volume={32},
      number={3},
       pages={557\ndash 561},
         url={http://projecteuclid.org/euclid.ijm/1255989005},
      review={\MR{947046}},
}

\bib{Selecta}{article}{
      author={Coulembier, Kevin},
       title={Tensor ideals, {D}eligne categories and invariant theory},
        date={2018},
        ISSN={1022-1824,1420-9020},
     journal={Selecta Math. (N.S.)},
      volume={24},
      number={5},
       pages={4659\ndash 4710},
         url={https://doi.org/10.1007/s00029-018-0433-z},
      review={\MR{3874701}},
}

\bib{AbEnv}{article}{
      author={Coulembier, Kevin},
       title={Monoidal abelian envelopes},
        date={2021},
        ISSN={0010-437X,1570-5846},
     journal={Compos. Math.},
      volume={157},
      number={7},
       pages={1584\ndash 1609},
         url={https://doi.org/10.1112/S0010437X21007399},
      review={\MR{4277111}},
}

\bib{PolyF}{article}{
      author={Coulembier, Kevin},
       title={Inductive systems of the symmetric group, polynomial functors and tensor categories},
        date={2024},
      eprint={2406.00892},
         url={https://arxiv.org/abs/2406.00892},
}

\bib{CPW}{article}{
      author={Carlson, Jon~F.},
      author={Peng, Chuang},
      author={Wheeler, Wayne~W.},
       title={Transfer maps and virtual projectivity},
        date={1998},
        ISSN={0021-8693,1090-266X},
     journal={J. Algebra},
      volume={204},
      number={1},
       pages={286\ndash 311},
         url={https://doi.org/10.1006/jabr.1997.7486},
      review={\MR{1623977}},
}

\bib{De2}{incollection}{
      author={Deligne, P.},
       title={Cat\'egories tensorielles},
        date={2002},
      volume={2},
       pages={227\ndash 248},
         url={https://doi.org/10.17323/1609-4514-2002-2-2-227-248},
        note={Dedicated to Yuri I. Manin on the occasion of his 65th birthday},
      review={\MR{1944506}},
}

\bib{De1}{incollection}{
      author={Deligne, P.},
       title={Cat\'egories tannakiennes},
        date={1990},
   booktitle={The {G}rothendieck {F}estschrift, {V}ol.\ {II}},
      series={Progr. Math.},
      volume={87},
   publisher={Birkh\"auser Boston, Boston, MA},
       pages={111\ndash 195},
      review={\MR{1106898}},
}

\bib{DT}{article}{
      author={Doty, Stephen},
      author={Henke, Anne},
       title={Decomposition of tensor products of modular irreducibles for {${\rm SL}_2$}},
        date={2005},
        ISSN={0033-5606,1464-3847},
     journal={Q. J. Math.},
      volume={56},
      number={2},
       pages={189\ndash 207},
         url={https://doi.org/10.1093/qmath/hah027},
      review={\MR{2143497}},
}

\bib{Go}{book}{
      author={Goss, David},
       title={Basic structures of function field arithmetic},
      series={Ergebnisse der Mathematik und ihrer Grenzgebiete (3) [Results in Mathematics and Related Areas (3)]},
   publisher={Springer-Verlag, Berlin},
        date={1996},
      volume={35},
        ISBN={3-540-61087-1},
         url={https://doi.org/10.1007/978-3-642-61480-4},
      review={\MR{1423131}},
}

\bib{Jantzen}{book}{
      author={Jantzen, Jens~Carsten},
       title={Representations of algebraic groups},
     edition={Second},
      series={Mathematical Surveys and Monographs},
   publisher={American Mathematical Society, Providence, RI},
        date={2003},
      volume={107},
        ISBN={0-8218-3527-0},
      review={\MR{2015057}},
}

\bib{STWZ}{article}{
      author={Sutton, Louise},
      author={Tubbenhauer, Daniel},
      author={Wedrich, Paul},
      author={Zhu, Jieru},
       title={{${\rm SL}_2$} tilting modules in the mixed case},
        date={2023},
        ISSN={1022-1824,1420-9020},
     journal={Selecta Math. (N.S.)},
      volume={29},
      number={3},
       pages={Paper No. 39, 40},
         url={https://doi.org/10.1007/s00029-023-00835-0},
      review={\MR{4587641}},
}

\end{biblist}
\end{bibdiv}
\bibliographystyle{amsrefs}%

\end{document}